\documentclass{article}%
\usepackage{amsmath}
\usepackage{amsfonts}
\usepackage{amssymb}
\usepackage{blindtext}
\usepackage{enumitem}
\usepackage{fullpage}
\usepackage{graphicx}
\usepackage{pstricks}
\usepackage{pstricks-add}
\usepackage{pst-node}
\usepackage{pst-plot}
\usepackage{tikz,verbatim}
\usetikzlibrary{shapes,arrows}
\usetikzlibrary{trees}
\providecommand{\U}[1]{\protect\rule{.1in}{.1in}}
\newtheorem{theorem}{Theorem}
\newtheorem{lemma}[theorem]{Lemma}

\newtheorem{definition}{Definition}
\newtheorem{notation}[definition]{Notation}

\newtheorem{example}[definition]{Example}

\newenvironment{proof}{\noindent{\em Proof:}}{$\Box$~\\}
\allowdisplaybreaks

\begin{document}

\title{ Optimality of Curtiss Bound on \\Poincar\'e Multiplier for \\Positive Univariate Polynomials \\}
\author{
Hoon Hong~\footnote{
Department of Mathematics, 
North Carolina State University,
Raleigh NC 27695 USA, 
hong@ncsu.edu}
and 
Brittany Riggs~\footnote{
Corresponding author.
Department of Mathematics,
Elon University,
Elon NC 27244 USA,
briggs@elon.edu}
}
\maketitle

\begin{abstract}
Let $f$ be a monic univariate polynomial with non-zero constant term. We say
that $f$ is \emph{positive\/} if~$f(x)$ is positive over all $x\geq0$. If all
the coefficients of $f$ are non-negative, then $f$ is trivially positive. In~
1883, Poincar\'e proved that$f$ is positive if and only if there exists a monic
polynomial $g$ such that all the coefficients of $gf$ are non-negative. Such
polynomial $g$ is called a \emph{Poincar\'e multiplier\/} for the positive
polynomial $f$. Of course one hopes to find a multiplier with smallest degree.
This naturally raised a challenge: find an upper bound on the smallest degree
of multipliers. In 1918, Curtiss provided such a bound. Curtiss also showed
that the bound is optimal (smallest)\ when degree of $f$ is 1 or 2. It is easy
to show that the bound is not optimal when degree of $f$ is higher. The
Curtiss bound is a simple expression that depends only on the angle (argument)
of non-real roots of $f$. In this paper, we show that the Curtiss bound is
optimal among all the bounds that depends only on the angles.

\end{abstract}

%###################################################################################

\section{Introduction}
Let $f$ be a monic univariate polynomial with non-zero constant term. We say
that $f$ is \emph{positive\/} if~$f(x)$ is positive over 
all $x\geq0$.~\footnote{
It is equivalent to the condition that $f$ has no positive real root.
In fact, most of the discussion in the paper has a counter-part in a real root counting
problem, namely checking whether $f$ has exactly $k$ positive real roots for a given $k$. There have been many works on the topic and on its extensions, just to list a few: \cite{D37,B07,F13,HS,S09,T11,CP19,FT22}.  
}
If all
the coefficients of $f$ are non-negative, then obviously $f$ is positive, but the converse is not true. Consider the following small (toy) examples:
\begin{align*}
f_{1} &=x^{4}+x^{3}+10x^{2}+2x+10\\
f_{2} &=x^{4}-x^{3}-10x^{2}-2x+10\\
f_{3} &=x^{4}-x^{3}+10x^{2}-2x+10
\end{align*}
\noindent Note that all the$\ $coefficients of $f_{1}$ are non-negative. Thus
it is trivial to see that $f_{1}$ is positive. However $f_{2}$ and $f_{3}$
have negative coefficients, and thus it is not obvious whether they are
positive or not. It turns out that $f_{2}$ is not positive and $f_{3}$ is positive.

In 1883, Poincar\'e \cite{P83} proposed and proved a modification of the converse:  if $f$ is positive, then there exists a monic
polynomial $g$ such that all the coefficients of $gf$ are non-negative. Such
polynomial $g$ is called a \emph{Poincar\'{e} multiplier} for the positive
polynomial $f$.  For the above examples, we have

\begin{itemize}
\item Since $f_{2}$ is not positive, there is no Poincar\'{e} multiplier for
$f_{2}$.

\item Since $f_{3}$ is positive, there is a Poincar\'{e} multiplier for
$f_{3}$. For instance, let $g=x+1$. Then
\[
gf_{3}=\left(  x+1\right)  \left(  x^{4}-x^{3}+10x^{2}-2x+10\right)
=x^{5}+9x^{73}+8x^{2}+8x+10
\]
Note that all the coefficients of $gf_{3}$ are non-negative.  
\end{itemize}

\noindent In 1928, P\'olya provided a shape for the Poincar\'{e} multiplier \cite{P28}, namely $(x+1)^k$.~\footnote{
P\'olya used the multivariate version
of the shape while trying to find a certificate for a variant of Hilbert's 17th problem \cite{H02}. For similar efforts on different variants of Hilbert's 17th problem,
see~\cite{A27,K64,S91,S04,P93,NS07}.
}
In fact, in the above example, we used a multiplier of the P\'olya shape with $k=1$.  However, there are other multipliers for
$f_{3}$ that do not have the P\'olya shape. For instance, let $g=x^{2}+x+1$.  Note that it does not have the P\'olya shape, but it is a Poincar\'{e} multiplier since all the coefficients of $gf_{3}$ are again non-negative:
\[
gf_{3}=\left(  x^{2}+x+1\right)  \left(  x^{4}-x^{3}+10x^{2}-2x+10\right)
=x^{6}+10x^{4}+7x^{3}+18x^{2}+8x+10.
\]
In fact, there are infinitely many multipliers of diverse shapes.
Hence  a  challenge naturally arises: \emph{Find an upper bound on the
smallest degree of multipliers}.

In 1918, Curtiss found an upper bound  \cite{C18}. However, it seems to have been forgotten until the paper was digitized.  Separately, in 2001 Powers and Reznick gave another bound based on the P\'olya shape in the multivariate homogeneous case \cite{PR01}.  Avenda\~{n}o (page 2885 of \cite{A10}) specialized this bound to the univariate case.   

It is easy to show that the Powers-Reznick bound is not optimal.  Curtiss showed that his bound is optimal when degree of $f$ is $1$
or $2$. It is, however, easy to show that the bound is not optimal when degree of $f$ is
higher. For example, the Curtiss bound for $f_{3}$ is $2$, which is bigger
than the optimal degree which is~$1$.  Hence another  challenge naturally arises: \emph{Find the optimal (smallest) bound}.

This can be done, in principle, by the following process: for each degree $d$, starting
with $0$, we check whether there exists a Poincar\'{e} multiplier of degree~$d$. If not, we continue with next degree. If yes, we report $d$ to be the optimal bound. The process is guaranteed to terminate due to Poincar\'e's  theorem. 
Checking the existence of a Poincar\'{e} multiplier of a given degree~$d$ can be
carried out by any algorithm (such as simplex phase I) for deciding the existence of a solution for a system of linear inequalities in the coefficients of the multiplier polynomial.

However, it has  two major difficulties. (1) It
requires that we run the software on each $d$ from scratch.
(2) Each run on degree $d$ quickly becomes very time-consuming as $d$ grows, making this approach practically infeasible, especially when  the  optimal bound is very large. 

Fortunately, there is a related challenge which might be feasible to tackle.  For this, we note the Curtiss bound (Theorem \ref{thm:bound}) depends only on the angles
(arguments) of the non-real roots of $f$. Thus, we formulate the following alternative challenge: {\it Find an optimal bound among the family of bounds that depend only on the angles of the non-real roots.}  

Note that Curtiss already showed that his bound is optimal among the family of those bounds for degrees~1 and~2.  One wonders whether the Curtiss bound is optimal among this family of bounds for higher degrees.  The main contribution of this paper is
to prove  that it is so (Theorem~\ref{thm:optimal}).

\section{Main Result}

Let $f\in\mathbb{R}\left[  x\right]  $ be monic such that $\underset{x\geq
0}{\forall}f\left(  x\right)  >0$, that is, $f$ does not have any non-negative
real root.  We first define a notation for the optimal degree of the Poincar\'{e} multiplier as it will be used throughout this work.

\begin{definition}
[Optimal Bound]The optimal degree of the Poincar\'{e} multiplier for $f$, written as $\operatorname*{opt}%
\left(  f\right)  $, is defined by%
\[
\operatorname*{opt}\left(  f\right)  =\min_{\substack{g\in\mathbb{R}\left[
x\right]  \backslash0\\\operatorname*{coeff}\left(  gf\right)  \geq0}}\deg(g).
\]

\end{definition}

\begin{example}\label{ex:main}
Consider the polynomial
$$f=\left(\prod_{i=1}^4\left(x^2-2(1)\cos(\theta_i)\, x+1^2\right)\right)(x+1)$$
where $\theta=\left\{\dfrac{7\pi}{24},\dfrac{10\pi}{24},\dfrac{11\pi}{24},\dfrac{14\pi}{24}\right\}$.
First, note that $\operatorname*{opt}\left(  f\right) \neq 0$, as $f$ has mixed sign coefficients. If we let the Poincar\'{e} multiplier be $g=x+1$, we see we have $\operatorname*{opt}\left(  f\right)=1$, since $\operatorname*{coeff}\left(  gf\right)  \geq0$.
\end{example}

\begin{theorem}
[Curtiss Bound 1918 \cite{C18}]\label{thm:bound}Let $r_{1}e^{\pm i\theta_{1}%
},\ldots,r_{\sigma}e^{\pm i\theta_{\sigma}}$ be the roots of $f$ where multiple
roots are repeated. Let%
\[
b\left(  f\right)  = \underset{\theta_i\neq \pi}{\sum_{1\leq i\leq \sigma}}\left(  \left\lceil \frac{\pi}{\theta_{i}%
}\right\rceil -2\right).
\]
Then $opt\left(  f\right)  \leq b\left(  f\right)  $ and the equality
holds if $\deg f\leq2$.
\end{theorem}

\begin{example}\label{ex:notopt}
Consider the same polynomial from Example \ref{ex:main},
$$f=\left(\prod_{i=1}^4\left(x^2-2(1)\cos(\theta_i)\, x+1^2\right)\right)(x+1)$$
where $\theta=\left\{\dfrac{7\pi}{24},\dfrac{10\pi}{24},\dfrac{11\pi}{24},\dfrac{14\pi}{24}\right\}$.The Curtiss bound for this polynomial is 
\begin{align*} b(f) & =\left(  \left\lceil \frac{\pi}{\frac{7\pi}{24}
}\right\rceil -2\right)+\left(  \left\lceil \frac{\pi}{\frac{10\pi}{24}
}\right\rceil -2\right)+\left(  \left\lceil \frac{\pi}{\frac{11\pi}{24}
}\right\rceil -2\right)+\left(  \left\lceil \frac{\pi}{\frac{14\pi}{24}
}\right\rceil -2\right) \\
& =2+1+1+0 \\
& =4. \end{align*}
Here we can see by comparing to Example \ref{ex:main} that the Curtiss bound is not always optimal.
\end{example}

\begin{theorem}
[Main Result: Angle-Based Optimality of Curtiss's Bound]\label{thm:optimal}We
have%
\[
\underset{\theta_{1},\ldots,\theta_{\sigma}\in(0,\pi]}{\forall}%
\ \ \ \underset{r_{1}%
,\ldots,r_{\sigma}>0}{\exists}\ \operatorname*{opt}\left(  f\right)  =b\left(
f\right).
\]

\end{theorem}

\begin{example}
Consider the following polynomial with the same angles as that from Example \ref{ex:notopt} but different radii,
$$f=\left(\prod_{i=1}^4\left(x^2-2(10^{i-1})\cos(\theta_i)\, x+(10^{i-1})^2\right)\right)(x+10^4)$$
where $\theta=\left\{\dfrac{7\pi}{24},\dfrac{10\pi}{24},\dfrac{11\pi}{24},\dfrac{14\pi}{24}\right\}$.
The Curtiss bound for this polynomial is still $b(f)=4$.  However, in this case we have $\operatorname*{opt}(f)=b(f)=4$.

\end{example}

\section{Proof of Angle-Based Optimality (Theorem \ref{thm:optimal})}
\subsection{Overview}
As the proof is rather lengthy, we will include an overview of the structure of the proof. First, we provide a diagram of the various components of the proof.  As shown in the diagram, the main claim is split into two sub-claims, which are then proven via several smaller claims.  Below, we explain the diagram from a big picture standpoint.  In the sections referenced in the diagram, we provide the proofs for each claim. 
\begin{center}
\tikzstyle{bag1} = [text width=16em, text centered]
\tikzstyle{bag2} = [text width=11em, text centered]
\begin{tikzpicture}
[edge from parent fork down]
\node[bag1]  {Optimality of Curtiss's Bound \\ Section \ref{sec:main}, Theorem \ref{thm:optimal}}[sibling distance = 8.5cm, level distance = 1.35cm]
    child {node[bag1] {All Angles in Quadrant 1 \\ Section \ref{sec:Q1}, Lemma \ref{lem:core}}[sibling distance = 3cm, level distance = 1.35cm]
    child {node[bag2] {One Less Degree \\ Lemma \ref{lem:deg}}}
    child {node[bag2] {3 Coefficients \\ Lemma \ref{lem:subset}}}
    child {node[bag2] {Existence of Radii \\ Lemma \ref{lem:q1m2}}} [sibling distance = 3.8cm, level distance = 1.35cm]
    child {node[bag2] {Partial Witness \\ Lemma \ref{lem:witness}}}}
    child {node[bag1] {Some Angles in Quadrant 2 \\ Section \ref{sec:q2}, Lemma \ref{lem:reduce}}[sibling distance = 2.5cm, level distance = 1.35cm]
    child {node[bag1] {One New Quadrant 2 Factor \\ Lemma \ref{lem:extension}}
    child {node[bag1] {Separation of Sets \\ Lemma \ref{lem:ch}}}}};
\end{tikzpicture}
\end{center}
Our goal is to prove the angle-based optimality of Curtiss's bound on the degree of the Poincar\'{e} multiplier for a positive polynomial.    \bigskip

\noindent Symbolically, we have
\[
\underset{\theta\in\Theta}{\forall}%
\ \ \underset{r\in R}{\exists}\ \ \operatorname*{opt}\left(  f\right)  =b\left(
f\right).
\]
After extensive exploration, we identified a natural split in strategies between collections of angles $0<\theta<\dfrac{\pi}{2}$ and those with additional angles $\dfrac{\pi}{2}\leq \theta \leq \pi$.  We now have
\[
\underset{\theta\in\Theta_1 \cup \Theta_2}{\forall
}\ \ \underset{r\in R}{\exists}\ \ \operatorname*{opt}\left(  f\right)
=b\left(  f\right).
\]
We first focus on the claim for $\theta\in\Theta_1$, polynomials with complex root arguments $0<\theta<\dfrac{\pi}{2}$.  That is, for any collection of angles in this subset, there exists a corresponding collection of radii such that Curtiss's bound is the minimum possible degree for the Poincar\'{e} multiplier to achieve a product with all non-negative coefficients. One way to show this is to prove that there exist radii such that any multiplier of smaller degree will produce a product with a negative coefficient.  We can then rewrite the claim as
\[\underset{\theta\in\Theta_1}{\forall}\ \ \underset{r\in R}{\exists}\ \ \underset{g \in G}{\forall}\ \ \underset{k \in K}{\exists}%
\ \  \operatorname*{coeff}(gf,k)<0.\]
Note, with this rewrite we consider a problem with three quantifier alternations.  As this is a very complex problem, we seek to reduce the search space for the negative coefficient.  It is straightforward to prove that, rather than consider the entire set of multipliers of smaller degree, it is sufficient to focus on the set of multipliers with degree one less than Curtiss's bound to show that any such multiplier will produce a product with a negative coefficient.  We now have
\[\underset{\theta\in\Theta_1}{\forall}\ \ \underset{r\in R}{\exists}\ \ \underset{g \in G'}{\forall}\ \ \underset{k \in K}{\exists}%
\ \ \operatorname*{coeff}(gf,k)<0\]
for a subset $G'$ of the original multiplier search space $G$ that is determined by the Curtiss bound (hence, by $\theta$).  After detailed analysis of examples, we surprisingly discovered that we can also narrow down the potential locations of the negative coefficient.  In fact, although the collection of product coefficients grows in size with Curtiss's bound, we can identify a negative coefficient from a subset of just 3 possibilities.  Hence
 \[\underset{\theta\in\Theta_1}{\forall}\ \ \underset{r\in R}{\exists}\ \ \underset{g \in G'}{\forall}\ \ \underset{k \in K'}{\exists}%
\ \ \operatorname*{coeff}(gf,k)<0\]
for a subset $K'$ of the original coefficient search space $K$ that is determined by the number of angles in $\theta$ and the Curtiss bound (hence, by $\theta$).  Finally, with more difficulty, we determined how to reduce the search space for the desired radii.  By identifying a witness, in terms of $\cos\left(\theta\right)$, for all but one of the radii, we reduced the radii search space to just a one-dimensional space for the final radius.  We have
\[\underset{\theta\in\Theta_1}{\forall}\ \ \underset{r\in R'}{\exists}\ \ \underset{g \in G'}{\forall}\ \ \underset{k \in K'}{\exists}%
\ \ \operatorname*{coeff}(gf,k)<0\]
for a subset $R'$ of the original radii search space $R$.  Proving this final version of the claim was still a significant challenge.  We examined the geometry of the product coefficients as a function of this final unknown radius in order to verify the claim.  Expressing the coefficients of the original polynomial $f$ in terms of the elementary symmetric polynomials in the roots of $f$ proved essential to our progress, enabling us to make claims about the sign of the coefficients that were otherwise quite challenging. \bigskip

\noindent Having proved the claim in the case where all angles fall in $\Theta_1$, we move on to the case where we have mixed angles from $\Theta_1$ or $\Theta_2$.  The crux of this claim is that any polynomial in which Curtiss's bound is optimal can be multiplied by a quadratic with angle $\dfrac{\pi}{2}\leq \theta < \pi$ and a radius that preserves the optimality of Curtiss's bound or a linear factor with an appropriate radius that preserves the optimality of Curtiss's bound.  Thus, for any collection of additional angles in quadrant 2, we can inductively prove the existence of the necessary radii to ensure the optimality of Curtiss's bound.
\[
\underset{\theta\in\Theta_2}{\forall
}\ \ \underset{r\in R}{\exists}\ \ \operatorname*{opt}\left( f_{\theta,r}\,f\right)  =\operatorname*{opt}\left(
f\right)=b(f)
\]
Note that, in this case, the additional factors in the polynomial do not increase the Curtiss bound. For this sub-claim we utilize a different strategy, relying on linear algebra to rewrite the optimality claim as a separation between two sets.  

\subsection{Proof of Angle-Based Optimality (Theorem \ref{thm:optimal})}\label{sec:main}
\noindent Since we must treat angles separately according to their measure, we split the collection of angles $\theta_1,\ldots,\theta_{\sigma}$ into three sets as detailed below.  Let%
\begin{align*}
f  &  =f_{\theta,r_{\theta}}\,f_{\phi,r_{\phi}}\,f_{\pi,p}\\
f_{\theta,r_{\theta}}  &  =\prod\limits_{\substack{1\leq i\leq\ell
\\0<\theta_{i}<\frac{\pi}{2}}}\left(  x^{2}-2r_{\theta_{i}}\cos\theta
_{i}\,x+r_{\theta_{i}}^{2}\right)  \ \ \ \text{where }r_{\theta_{i}%
}>0\ \text{and }0<\theta_{1}\leq\cdots\leq\theta_{\ell}<\dfrac{\pi}{2} \\
f_{\phi,r_{\phi}}  &  =\prod\limits_{\substack{1\leq i\leq k\\\frac{\pi}%
{2}\leq\phi_{i}<\pi}}\left(  x^{2}-2r_{\phi_{i}}\cos\phi_{i}\,x+r_{\phi_{i}%
}^{2}\right)  \ \text{where }r_{\phi_{i}}>0~\text{and }\ \dfrac{\pi}{2}\leq\phi_{1}%
\leq\cdots\leq\phi_{k} < \pi\\
f_{\pi,p}  &  =\prod\limits_{1\leq i\leq t}\left(  x+p_{i}\right)
\ \,\text{where }p_{i}>0
\end{align*}
where $k+\ell+t=\sigma$  and $2(k+\ell)+t=n=\deg
(f)$.  \bigskip

\begin{proof}
[Proof of Theorem \ref{thm:optimal}] 
We need to show
\[
\underset{\frac{\pi}{2} \leq\phi_{1}\leq\cdots\leq\phi_{k}<\pi}{\forall
}\ \ \underset{0<\theta_{1}\leq\cdots\leq\theta_{\ell}<\frac{\pi}{2}}{\forall
}\ \ \underset{p_{1},\ldots,p_{t}>0}{\exists}\ \ \underset{r_{\phi_{1}}%
,\ldots,r_{\phi_{k}}>0}{\exists}\ \ \underset{r_{\theta_{1}},\ldots
,r_{\theta_{\ell}}>0}{\exists}\ \ \operatorname*{opt}\left(  f\right)
=b\left(  f\right).
\]
Let $\dfrac{\pi}{2}\leq\phi_{1}%
\leq\cdots\leq\phi_{k} < \pi\ \,$and
$0<\theta_{1}\leq\cdots\leq\theta_{\ell}<\dfrac{\pi}{2}$ be arbitrary but
fixed. We need to show
\[
\underset{p_{1},\ldots,p_{t}>0}{\exists}\ \ \underset{r_{\phi_{1}}%
,\ldots,r_{\phi_{k}}>0}{\exists}\ \ \underset{r_{\theta_{1}},\ldots
,r_{\theta_{\ell}}>0}{\exists}\ \ \operatorname*{opt}\left(  f\right)
=b\left(  f\right)  .
\]
We need to find a witness for $p,r_{\phi},r_{\theta}$ such that
$\operatorname*{opt}\left(  f\right)  =b\left(  f\right)  $.
We propose a witness candidate as follows.
From the All Angles in Quadrant 1 Lemma (Lemma \ref{lem:core}), for the fixed $\theta$, we have%
\begin{equation}
\underset{r_{\theta_{1}},\ldots,r_{\theta_{\ell}}>0}{\exists}%
\ \ \ \operatorname*{opt}\left(  f_{\theta,r_{\theta}}\right)  =b\left(
f_{\theta,r_{\theta}}\right). \label{l:core}%
\end{equation}
From the Some Angles in Quadrant 2 Lemma (Lemma \ref{lem:reduce}), for the fixed $\phi$ and $\theta$, we have
\begin{equation}
\underset{p_{1},\ldots,p_{t}>0}{\exists}\ \ \ \underset{r_{\phi_{1}}%
,\ldots,r_{\phi_{k}}>0}{\exists}\ \ \ \underset{r_{\theta_{1}},\ldots
,r_{\theta_{\ell}}>0}{\forall}\ \ \ \operatorname*{opt}\left( f_{\theta,r_{\theta}}\,f_{\phi,r_{\phi}}\,f_{\pi,p}\right)  =\operatorname*{opt}%
\left(  f_{\theta,r_{\theta}}\right).  \label{l:reduce}%
\end{equation}
We propose $p,r_{\phi},r_{\theta}$ appearing in the above two facts as a
witness candidate.

\bigskip

We verify that the proposed candidate is indeed a witness, that is,
$\operatorname*{opt}\left(  f\right)  =b\left(  f\right)  $.
Note%
\begin{align*}
\operatorname*{opt}\left(  f\right)   &  =\operatorname*{opt}\left( f_{\theta,r_{\theta}}\,f_{\phi,r_{\phi}}\,f_{\pi,p}\right) \\
&  =\operatorname*{opt}\left(  f_{\theta,r_{\theta}}\right)  \ \ \text{by
(\ref{l:reduce})}\\
&  =b\left(  f_{\theta,r_{\theta}}\right)  \ \text{by (\ref{l:core})}\\
&  =b\left(
f_{\theta,r_{\theta}}\right)  +b\left(  f_{\phi,r_{\phi}}\right)  +b\left(  f_{\pi,p}\right)  \ \text{since }
b\left(  f_{\phi,r_{\phi}}\right)=b\left(  f_{\pi,p}\right)  =0\ \\
&  =b\left( f_{\theta,r_{\theta}}\,f_{\phi,r_{\phi}}\,f_{\pi,p}\right) \\
&  =b\left(  f\right).
\end{align*}
\noindent Hence the theorem is proved.
\end{proof}

\subsection{Common Notations}

The following notation and lemma will be used in both the subsequent subsections.  Let%
\begin{align*}
f  &  =a_{n}x^{n}+\cdots+a_{0}x^{0}\\
g  &  =b_{s}x^{s}+\cdots+b_{0}x^{0}%
\end{align*}
where $a_{n}=1$ and $b_{s}=1$. We first rewrite them using vectors. Let%
\[
a=\left[
\begin{array}
[c]{ccc}%
a_{0} & \cdots & a_{n}%
\end{array}
\right]  \ \ \ \ \ \ \ b=\left[
\begin{array}
[c]{ccc}%
b_{0} & \cdots & b_{s}%
\end{array}
\right]
\]
and let%
\[
x_{k}=\left[
\begin{array}
[c]{c}%
x^{0}\\
\vdots\\
x^{k}%
\end{array}
\right].
\]
Then we can write $f$ and $g$ compactly as%

\[
f=ax_{n}\ \ \ \ \text{and} \ \ \ \ g=bx_{s}.%
\]
Let%
\[
A_{s}=\left[
\begin{array}
[c]{cccccc}%
a_{0} & \cdots & \cdots & a_{n} &  & \\
& \ddots &  &  & \ddots & \\
&  & a_{0} & \cdots & \cdots & a_{n}%
\end{array}
\right]  \in\mathbb{R}^{\left(  s+1\right)  \times(s+n+1)}.%
\]

\begin{lemma}[Coefficients]
\label{lem:bA}$\operatorname*{coeffs}\left(  gf\right)  =bA_{s}$
\end{lemma}

\begin{proof}
Note%
\begin{align*}
gf  &  =\left(  bx_{s}\right)  \left(  ax_{n}\right) \\
&  =b\left(  x_{s}ax_{n}\right) \\
&  =b\left[
\begin{array}
[c]{c}%
x^{0}\\
\vdots\\
x^{s}%
\end{array}
\right]  \left[
\begin{array}
[c]{ccc}%
a_{0} & \cdots & a_{n}%
\end{array}
\right]  \left[
\begin{array}
[c]{c}%
x^{0}\\
\vdots\\
x^{n}%
\end{array}
\right] \\
&  =b\left[
\begin{array}
[c]{cccccc}%
a_{0} & \cdots & \cdots & a_{n} &  & \\
& \ddots &  &  & \ddots & \\
&  & a_{0} & \cdots & \cdots & a_{n}%
\end{array}
\right]  \left[
\begin{array}
[c]{c}%
x^{0}\\
\vdots\\
x^{s+n}%
\end{array}
\right] \\
&  =bA_{s}x_{s+n}%
\end{align*}
Hence $\operatorname*{coeffs}\left(  gf\right)  =bA_{s}$.
\end{proof}

\subsection{Concerning Angles in Quadrant 1, $0<\theta<\dfrac
{\pi}{2}$}\label{sec:Q1}

First, we present the overarching argument for the claim that, given a collection of angles in quadrant 1 (excluding the axes), there exist radii such that the Curtiss bound is optimal.  There are a number of claims required for this proof.
\begin{center}
\tikzstyle{bag1} = [text width=16em, text centered]
\tikzstyle{bag2} = [text width=11em, text centered]
\begin{tikzpicture}
[edge from parent fork down]
\node[bag1] {All Angles in Quadrant 1 \\ Lemma \ref{lem:core}}[sibling distance = 3cm, level distance = 1.35cm]
    child {node[bag2] {One Less Degree \\ Lemma \ref{lem:deg}}}
    child {node[bag2] {3 Coefficients \\ Lemma \ref{lem:subset}}}
    child {node[bag2] {Existence of Radii \\ Lemma \ref{lem:q1m2}}} [sibling distance = 3.8cm, level distance = 1.35cm]
    child {node[bag2] {Partial Witness \\ Lemma \ref{lem:witness}}};
\end{tikzpicture}
\end{center} 
We will utilize the following notations throughout this subsection.  For $\ell \geq 2$, let
\begin{align*}
\alpha_{i}  &  =r_{i}e^{i\theta_{i}}\ \ \ \text{for }1\leq i\leq\ell\\
\alpha_{\ell+i}  &  =r_{i}e^{-i\theta_{i}}\ \ \ \text{for }1\leq i\leq\ell\\
t_{i}  &  =\cos\left(\theta_{i}\right)\\
 f_{\theta
,r_{\theta}}  &  =\prod_{i=1}^{\ell}\left(  x^{2}-2r_{i}t_{i}x+r_{i}^{2}\right)
=\prod_{i=1}^{\ell}(x-\alpha_{i})(x-\alpha_{\ell+i})=\sum_{i=0}^{2\ell}%
a_{i}x^{i}\\
s  &  =b( f_{\theta
,r_{\theta}})\\
g  &  =x^{s-1}+b_{s-2}x^{s-2}+\cdots+b_{1}x+b_{0}\\
c_{k}  &  =\operatorname*{coeff}(g\, f_{\theta
,r_{\theta}},x^{k}).
\end{align*}
Note that $a_{i}=(-1)^{2\ell-i}e_{2\ell-i}\left(  \alpha_{1},\hdots,\alpha_{2\ell
}\right)  $ where $e_{k}\left(  \alpha_{1},\hdots,\alpha_{2\ell}\right)  $ is
the elementary symmetric polynomial of degree $k$ in the roots $\alpha
_{1},\hdots,\alpha_{2\ell}$. When $\ell=0$, we define $e_{0}=1$.  Second, $t_{i}> 0$ since $0<\theta_{i} <\dfrac{\pi}{2}$.

\begin{lemma}[All Angles in Quadrant 1]
\label{lem:core}%
\[
\underset{\ell\geq0}{\forall}\ \ \underset{0<\theta_{1}\leq
\ldots\leq\theta_{\ell}<\frac{\pi}{2}}{\forall}\ \ \underset{r_{\theta_{1}},\ldots
,r_{\theta_{\ell}}>0}{\exists}\ \operatorname*{opt}\left(  f_{\theta
,r_{\theta}}\right)  =b\left(  f_{\theta,r_{\theta}}\right)
\]

\end{lemma}

\begin{proof}
We need to prove the following claim for every $\ell\geq0$.
\[
\underset{0<\theta_{1}\leq
\ldots\leq\theta_{\ell}<\frac{\pi}{2}}{\forall
}\ \ \underset{r_{\theta_{1}},\ldots,r_{\theta_{\ell}}>0}{\exists
}\ \operatorname*{opt}\left(  f_{\theta,r_{\theta}}\right)  =s
\]
where $s=b(  f_{\theta,r_{\theta}})$.
By the One Less Degree Lemma (Lemma \ref{lem:deg}), it suffices to show
\[%
\begin{array}
[c]{cl}
& \underset{0<\theta_{1}\leq
\ldots\leq\theta_{\ell}<\frac{\pi}{2}}{\forall
}\ \ \underset{r_{\theta_{1}},\ldots,r_{\theta_{\ell}}>0}{\exists
}\ \ \underset{\deg(g)=s-1}{\underset{g\in\mathbb{R}[x]}{\forall}%
}\ \ \underset{0 \leq k \leq2\ell+s-1}{\exists}\ \ c_{k} <0.
\end{array}
\]

\noindent We will show it in three cases depending on the values of $\ell$.
\begin{description}[align=left,labelwidth=1em]
\item[\sf Case:] $\ell=0$. 

Here, $f=1$. The claim is trivially true since $\operatorname*{opt}\left(
f_{\theta,r_{\theta}}\right)  =b\left(  f_{\theta,r_{\theta}}\right)  =0$.

\item[\sf Case:] $\ell=1$.

Immediate from the Curtiss Bound (Theorem \ref{thm:bound}).

\item[\sf Case:] $\ell\geq 2$.

Note
\begin{align*} &\ \underset{0<\theta_{1}\leq
\ldots\leq\theta_{\ell}<\frac{\pi}{2}}{\forall}\ \ \underset{r_{\theta_{1}},\ldots
,r_{\theta_{\ell}}>0}{\exists}\ \operatorname*{opt}\left(  f_{\theta
,r_{\theta}}\right)  =b\left(  f_{\theta,r_{\theta}}\right) \\
\iff &\ \underset{0<\theta_{1}\leq
\ldots\leq\theta_{\ell}<\frac{\pi}{2}}{\forall}\ \ \underset{r_{\theta_{1}},\ldots
,r_{\theta_{\ell}}>0}{\exists}\ \ \underset{\deg(g)<s}{\underset{g\in\mathbb{R}[x]}{\forall}%
}\ \ \underset{k}{\exists}\ \ c_k<0 \\
\Longleftarrow\  &\ \underset{0<\theta_{1}\leq
\ldots\leq\theta_{\ell}<\frac{\pi}{2}}{\forall}\ \ \underset{r_{\theta_{1}},\ldots
,r_{\theta_{\ell}}>0}{\exists}\ \ \underset{\deg(g)=s-1}{\underset{g\in\mathbb{R}[x]}{\forall}%
}\ \ \underset{k}{\exists}\ \ c_k<0 \ \ \ \ \text{by Lemma } \ref{lem:deg} \\
\Longleftarrow\  &\ \underset{0<\theta_{1}\leq
\ldots\leq\theta_{\ell}<\frac{\pi}{2}}{\forall}\ \ \underset{r_{\theta_{1}},\ldots
,r_{\theta_{\ell}}>0}{\exists}\ \ \underset{\deg(g)=s-1}{\underset{g\in\mathbb{R}[x]}{\forall}%
}\ \ \underset{k \in \{s-2,\, 2\ell+s-5,\, 2\ell+s-2\}}{\exists}\ \ c_k<0 \ \ \ \ \text{by Lemma } \ref{lem:subset}.  \end{align*}
We prove this final claim to be true in the Existence of Radii Lemma, Lemma \ref{lem:q1m2}.
\end{description}
\end{proof}

\noindent Now we prove Lemmas \ref{lem:deg} and \ref{lem:subset}, cited above, and also several other lemmas required for the main claim, the Existence of Radii Lemma (Lemma \ref{lem:q1m2}).  \bigskip

\noindent First, we prove the small but essential claim that if there is no Poincar\'{e} multiplier, $g$, of a certain degree, then there is also no Poincar\'{e} multiplier with smaller degree. \medskip

\begin{lemma}[One Less Degree]
\label{lem:deg} $\underset{g,\, \deg(g)=s}{\nexists}\ \ \underset{k}{\forall
}\ \ c_{k}\geq0\ \ \ \ \Longrightarrow\ \ \ \ \underset{g,\, \deg
(g)<s}{\nexists}\ \ \underset{k}{\forall}\ \ c_{k}\geq0$
\end{lemma}

\begin{proof}
We will prove via the contrapositive:
\[
\underset{g,\, \deg(g)<s}{\exists}\ \ \underset{k}{\forall}\ \ c_{k}%
\geq0\ \ \ \ \Longrightarrow\ \ \ \ \underset{g,\, \deg(g)=s}{\exists
}\ \ \underset{k}{\forall}\ \ c_{k}\geq0.
\]
Assume $\underset{g,\, \deg(g)<s}{\exists}\ \ \underset{k}{\forall
}\ \ c_{k}\geq0$. Then there exists a $g$ with $\deg(g)=t<s$ such that $g\,  f_{\theta
,r_{\theta}}$
has all non-negative coefficients. Let $u=s-t$. \bigskip

Consider the multiplier $x^{u}\, g$. Note that $(x^{u}\, g)\,  f_{\theta
,r_{\theta}}=x^{u}(g\,  f_{\theta
,r_{\theta}})$
must have all non-negative coefficients and $\deg(x^{u}\, g)=u+t=s$.  Then there exists a multiplier, $x^{u}\, g$ with degree equal to $s$
such that the product has all non-negative coefficients. \bigskip

Hence we have
\[
\underset{g,\, \deg(g)<s}{\exists}\ \ \underset{k}{\forall}\ \ c_{k}%
\geq0\ \ \ \ \Longrightarrow\ \ \ \ \underset{g,\, \deg(g)=s}{\exists
}\ \ \underset{k}{\forall}\ \ c_{k}\geq0.
\]

\end{proof}

\noindent Our ultimate goal for this section is to prove that there exist radii such that there is no multiplier with degree less than $b(  f_{\theta,r_{\theta}})$.  The One Less Degree Lemma (Lemma \ref{lem:deg}) states that it suffices to show there is no multiplier with degree $b(  f_{\theta,r_{\theta}})-1$. 
Next, we show that if we consider a multiplier of degree one less, it suffices to examine a specific subset of the coefficients in the product rather than all coefficients to prove the claim.  Note that the surprisingly small subset of coefficients was selected by examining the geometry of the coefficients and determining a minimally sufficient set.  This is the proof that a subset will suffice.  The proof that this particular subset satisfies the claim is given in Lemma \ref{lem:q1m2}.

\begin{lemma}[3 Coefficients]\label{lem:subset}
\begin{align*} 
& \underset{0<\theta_{1}\leq
\ldots\leq\theta_{\ell}<\frac{\pi}{2}}{\forall}\ \ \underset{r_{\theta_{1}},\ldots
,r_{\theta_{\ell}}>0}{\exists}\ \ \underset{\deg(g)=s-1}{\underset{g\in\mathbb{R}[x]}{\forall}%
}\ \ \underset{k \in \{s-2,\, 2\ell+s-5,\, 2\ell+s-2\}}{\exists}\ \ c_k<0 \\
\Longrightarrow \ \ \ \  & \underset{0<\theta_{1}\leq
\ldots\leq\theta_{\ell}<\frac{\pi}{2}}{\forall}\ \ \underset{r_{\theta_{1}},\ldots
,r_{\theta_{\ell}}>0}{\exists}\ \ \underset{\deg(g)=s-1}{\underset{g\in\mathbb{R}[x]}{\forall}%
}\ \ \underset{k}{\exists}\ \ c_k<0
\end{align*}
\end{lemma}

\begin{proof}
Recall that
\begin{align*} s & =b\left(  f_{\theta,r_{\theta}}\right) \\
g & =x^{s-1}+b_{s-2}x^{s-2}+\cdots+b_{1}x+b_{0} \\
t_i & =\cos(\theta_i).
\end{align*}
Note that if $\deg(f)=2\ell$ and $\deg(g)=s-1$, then $\deg(g\, f)=2\ell+s-1$.
\begin{align*}   &\ \underset{0<\theta_{1}\leq
\ldots\leq\theta_{\ell}<\frac{\pi}{2}}{\forall}\ \ \underset{r_{\theta_{1}},\ldots
,r_{\theta_{\ell}}>0}{\exists}\ \ \underset{\deg(g)=s-1}{\underset{g\in\mathbb{R}[x]}{\forall}%
}\ \ \underset{k}{\exists}\ \ c_k<0 \\
\iff &\ \underset{0<\theta_{1}\leq
\ldots\leq\theta_{\ell}<\frac{\pi}{2}}{\forall}\ \ \underset{r_{\theta_{1}},\ldots
,r_{\theta_{\ell}}>0}{\exists}\ \ \underset{\deg(g)=s-1}{\underset{g\in\mathbb{R}[x]}{\forall}%
}\ \ \underset{0 \leq k \leq 2\ell+s-1}{\exists}\ \ c_k<0 \\ 
\Longleftarrow\ \,   &\ \underset{0<\theta_{1}\leq
\ldots\leq\theta_{\ell}<\frac{\pi}{2}}{\forall}\ \ \underset{r_{\theta_{1}},\ldots
,r_{\theta_{\ell}}>0}{\exists}\ \ \underset{\deg(g)=s-1}{\underset{g\in\mathbb{R}[x]}{\forall}%
}\ \ \underset{k \in \{s-2,\, 2\ell+s-5,\, 2\ell+s-2\}}{\exists}\ \ c_k<0. \end{align*}
\end{proof}

\noindent The next two lemmas are helpful in the proof of Lemma \ref{lem:q1m2} and are listed separately here to streamline the proof of Lemma \ref{lem:q1m2}.

\begin{lemma}[Coefficient Expressions]\label{lem:coeffexp} We have
\begin{align*}
c_{s-2}  &  =a_{0}b_{s-2}+a_{1}b_{s-3}+\sum_{i=2}^{s-2}a_{i}b_{(s-2)-i}\\
c_{2\ell+s-5}  
&  =a_{2\ell-4}+a_{2\ell-3}b_{s-2}+a_{2\ell-2}b_{s-3}+a_{2\ell-1}%
b_{s-4}+b_{s-5}\\
c_{2\ell+s-2}  
&  =a_{2\ell-1}+b_{s-2}%
\end{align*}
where
\begin{align*}
\ell  &  \geq 2 \\
 f_{\theta
,r_{\theta}}  &  =\prod_{i=1}^{\ell}\left(  x^{2}-2r_{i}t_{i}x+r_{i}^{2}\right)
=\prod_{i=1}^{\ell}(x-\alpha_{i})(x-\alpha_{\ell+i})=\sum_{i=0}^{2\ell}%
a_{i}x^{i}\\
s  &  =b( f_{\theta
,r_{\theta}})\\
g  &  =x^{s-1}+b_{s-2}x^{s-2}+\cdots+b_{1}x+b_{0}\\
c_{k}  &  =\operatorname*{coeff}(g\, f_{\theta
,r_{\theta}},x^{k}).
\end{align*}
\end{lemma}
\begin{proof} Recall from the Coefficients Lemma (Lemma \ref{lem:bA}) for $0 \leq i \leq2\ell$ and $0 \leq j \leq
s-1$,
\begin{align*}
\operatorname*{coeffs}(gf)  &  =\left[
\begin{array}
[c]{ccc}%
b_{0} & \cdots & b_{s-1}%
\end{array}
\right]  \left[
\begin{array}
[c]{cccccc}%
a_{0} & \cdots & \cdots & a_{2\ell} &  & \\
& \ddots &  &  & \ddots & \\
&  & a_{0} & \cdots & \cdots & a_{2\ell}%
\end{array}
\right]
\end{align*}
Then
\begin{align*}
\operatorname*{coeff}(gf,x^{k})  &  =\sum_{i+j=k} a_{i}b_{j}\\
&  =\sum_{i=0}^{k} a_{i}b_{k-i}.%
\end{align*}
Hence, \[
\operatorname*{coeff}(gf,x^{k})=\displaystyle\sum_{i+j=k}
a_{i}b_{j}=\displaystyle\sum_{i=0}^{k} a_{i}b_{k-i}\;\; \text{for}\;\; 0 \leq k \leq2
\ell+s-1.\] 
\bigskip

\noindent Then we have
\begin{align*}
c_{s-2}  &  =\sum_{i=0}^{s-2} a_{i}b_{(s-2)-i}\\
&  =a_{0}b_{s-2}+ a_{1}b_{s-3}+\sum_{i=2}^{s-2}a_{i}b_{(s-2)-i}\\
c_{2\ell+s-5}  &  =\sum_{i+j=2\ell+s-5} a_{i}b_{j}\\
&  =a_{2\ell-4}b_{s-1}+a_{2\ell-3}b_{s-2}+a_{2\ell-2}b_{s-3}+a_{2\ell
-1}b_{s-4}+a_{2\ell}b_{s-5}\\
&  =a_{2\ell-4}+a_{2\ell-3}b_{s-2}+a_{2\ell-2}b_{s-3}+a_{2\ell-1}b_{s-4}+b_{s-5}%
\ \ \ \ \text{since } a_{2\ell}=b_{s-1}=1\\
c_{2\ell+s-2}  &  =\sum_{i+j=2\ell+s-2} a_{i}b_{j}\\
&  =a_{2\ell-1}b_{s-1}+ a_{2\ell}b_{s-2}\\
&  \ =a_{2\ell-1}+b_{s-2} \ \ \ \ \text{since } a_{2\ell}=b_{s-1}=1.
\end{align*}
\end{proof}

\noindent The following lemma verifies a property of elementary symmetric polynomials in the roots $\alpha_{1},\hdots,\alpha_{2\ell}$ that is required for the proof of the Existence of Radii Lemma (Lemma \ref{lem:q1m2}).
\begin{lemma}[Symmetric Polynomials]
\label{lem:epos} If $0<\theta_i<\dfrac{\pi}{2}$, then $e_{k}\left(
\alpha_{1},\hdots,\alpha_{2\ell}\right)  >0$ for $k=0,\ldots,2\ell$.
\end{lemma}

\begin{proof}
We will induct on $\ell$, the number of quadratic factors of $f$ with non-real roots.

\begin{description}[align=left,labelwidth=1em]
\item[\sf Base Case:] $\ell=0$. 

Note that $e_{0}=1>0$.

\item[\sf Hypothesis:] $\ell \geq 1$.

Assume $e_{k}\left(  \alpha_{1},\hdots,\alpha_{2\ell
}\right)  >0$ for $\ell\geq1$ and $0 \leq k \leq2\ell$.

\item[\sf Induction Step:] 

Prove $e_{k}\left(  \alpha_{1},\hdots,\alpha_{2(\ell
+1)}\right)  >0$ for $0 \leq k \leq2(\ell+1)$.

Let $f_{\ell}=\displaystyle\prod_{i=1}^{\ell}\left(  x^{2}-2r_{i}t_{i}%
x+r_{i}^{2}\right)  $ and $a_{\ell,k}=\operatorname*{coeff}(f_{\ell},x^{k})$.
Note that
\begin{align*}
f_{\ell+1}\  &  \ =\left(  x^{2}-2r_{\ell+1}t_{\ell+1}x+r_{\ell+1}^{2}\right)
f_{\ell}\\
a_{\ell+1,k}\  &  \ =a_{\ell,k-2}-2r_{\ell+1,t+1}a_{\ell,k-1}+r_{\ell+1}%
^{2}a_{\ell,k}.%
\end{align*}
Then
\begin{align*}
a_{\ell+1,k}  &  =(-1)^{2(\ell+1)-k}e_{2(\ell+1)-k}\left(  \alpha
_{1},\hdots,\alpha_{2(\ell+1)}\right) \\
&  =(-1)^{2\ell-(k-2)}e_{2\ell-(k-2)}\left(  \alpha_{1},\hdots,\alpha_{2\ell
}\right) \\
&  \ \ \ \ -2r_{\ell+1}t_{\ell+1}(-1)^{2\ell-(k-1)}e_{2\ell-(k-1)}\left(
\alpha_{1},\hdots,\alpha_{2\ell}\right) \\
&  \ \ \ \ +r_{\ell+1}^{2}(-1)^{2\ell-k}e_{2\ell-k}\left(  \alpha
_{1},\hdots,\alpha_{2\ell}\right).
\end{align*}
Hence, by dividing the above by $\left(-1\right)^{2(\ell+1)-k}$, we have
\begin{align*}
e_{2(\ell+1)-k}\left(  \alpha_{1},\hdots,\alpha_{2(\ell+1)}\right)   &
=e_{2\ell-(k-2)}\left(  \alpha_{1},\hdots,\alpha_{2\ell}\right) \\
&  \ \ \ \ -2r_{\ell+1}t_{\ell+1}(-1)^{-1}e_{2\ell-(k-1)}\left(  \alpha
_{1},\hdots,\alpha_{2\ell}\right) \\
&  \ \ \ \ +r_{\ell+1}^{2}(-1)^{-2}e_{2\ell-k}\left(  \alpha_{1}%
,\hdots,\alpha_{2\ell}\right) \\
&  =e_{2\ell-(k-2)}\left(  \alpha_{1},\hdots,\alpha_{2\ell}\right) \\
&  \ \ \ \ +2r_{\ell+1}t_{\ell+1}e_{2\ell-(k-1)}\left(  \alpha_{1}%
,\hdots,\alpha_{2\ell}\right) \\
&  \ \ \ \ +r_{\ell+1}^{2}e_{2\ell-k}\left(  \alpha_{1},\hdots,\alpha_{2\ell
}\right) \\
&  >0
\end{align*}
by the inductive hypothesis and the fact that $r_{\ell+1},t_{\ell+1}>0$.
\end{description}
\end{proof} \medskip

\noindent The next lemma provides witnesses for a claim that is utilized in the proof of the Existence of Radii Lemma (Lemma \ref{lem:q1m2}).  It is stated separately to preserve the intuitive flow of the proof of the Existence of Radii Lemma. \medskip

\begin{lemma}[Partial Witness]\label{lem:witness} $\underset{\ell \geq 3}{\forall}\ \ \underset{1>t_{1}\geq\cdots\geq t_{\ell
}>0}{\forall}\ \ \underset{r_{1},\ldots,r_{\ell}>0}{\exists}%
\ \ C(r)<1$
where 
$$C(r)=\dfrac{e_{0}\left(  \alpha_{1},\ldots,\alpha_{\ell-1},\alpha
_{\ell+1},\ldots,\alpha_{2\ell-1}\right)  \, e_{2\ell-2}\left( \alpha_{1},\ldots,\alpha_{\ell-1},\alpha
_{\ell+1},\ldots,\alpha_{2\ell-1}\right)
}{e_{1}\left(\alpha_{1},\ldots,\alpha_{\ell-1},\alpha
_{\ell+1},\ldots,\alpha_{2\ell-1}\right)  \, e_{2\ell-3}\left( \alpha_{1},\ldots,\alpha_{\ell-1},\alpha
_{\ell+1},\ldots,\alpha_{2\ell-1}\right)  }.$$  \end{lemma}

\begin{proof} Note
\begin{align*}
e_{0}\left( \alpha_{1},\ldots,\alpha_{\ell-1},\alpha
_{\ell+1},\ldots,\alpha_{2\ell-1}\right)   &  =1\\
e_{1}\left( \alpha_{1},\ldots,\alpha_{\ell-1},\alpha
_{\ell+1},\ldots,\alpha_{2\ell-1}\right)   &  =\alpha_{1}+\cdots+\alpha_{\ell-1}+\alpha
_{\ell+1}+\cdots+\alpha_{2\ell-1}\\
&  =\left(  \alpha_{1}+\alpha_{\ell+1}\right)  +\cdots+\left(  \alpha_{\ell-1
}+\alpha_{2\ell-1}\right) \\
&  =2r_{1}t_{1}+\cdots+2r_{\ell-1}t_{\ell-1}\\
e_{2\ell-3}\left(\alpha_{1},\ldots,\alpha_{\ell-1},\alpha
_{\ell+1},\ldots,\alpha_{2\ell-1}\right)   &  =\sum_{i \in\{1,\ldots,\ell-1
,\ell+1,\ldots,2\ell-1\}} \left(  \prod_{j \neq i} \alpha_{j} \right) \\
&  =\sum_{1\leq i \leq\ell-1} \left(  \prod_{j \neq i} \alpha_{j}+\prod_{j
\neq\ell+i} \alpha_{j}\right) \\
&  =\sum_{1\leq i \leq\ell-1} \left(  \left(  \alpha_{\ell+i}+\alpha
_{i}\right)  \prod_{j \neq i,\ell+i} \alpha_{j}\right) \\
&  = \sum_{1\leq i \leq\ell-1} \left(  \left(  2r_{i}t_{i}\right)  \prod_{j
\neq i} r_{j}^{2}\right) \\
e_{2\ell-2}\left( \alpha_{1},\ldots,\alpha_{\ell-1},\alpha
_{\ell+1},\ldots,\alpha_{2\ell-1}\right)   &  =\alpha_{1}\cdots\alpha_{\ell-1}%
\alpha_{\ell+1}\cdots\alpha_{2\ell-1}\\
&  =\alpha_{1}\alpha_{\ell+1} \cdots\alpha_{\ell-1}\alpha_{2\ell-1}\\
&  =r_{1}^{2} \cdots r_{\ell-1}^{2}.%
\end{align*}
Then
\begin{align*}
C(r)  &  =\dfrac{e_{0}\left( \alpha_{1},\ldots,\alpha_{\ell-1},\alpha
_{\ell+1},\ldots,\alpha_{2\ell-1}\right)  \, e_{2\ell-2}\left( \alpha_{1},\ldots,\alpha_{\ell-1},\alpha
_{\ell+1},\ldots,\alpha_{2\ell-1}\right)
}{e_{1}\left( \alpha_{1},\ldots,\alpha_{\ell-1},\alpha
_{\ell+1},\ldots,\alpha_{2\ell-1}\right)  \, e_{2\ell-3}\left( \alpha_{1},\ldots,\alpha_{\ell-1},\alpha
_{\ell+1},\ldots,\alpha_{2\ell-1}\right)  }\\
&  =\dfrac{r_{1}^{2} \cdots r_{\ell-1}^{2}}{\left(  2r_{1}t_{1}+\cdots
+2r_{\ell-1}t_{\ell-1}\right)  \left(  \sum_{1\leq i \leq\ell-1} \left(
\left(  2r_{i}t_{i}\right)  \prod_{j \neq i} r_{j}^{2}\right)  \right)  }\\
&  =\dfrac{r_{1}\cdots r_{\ell-1}}{\left(  2r_{1}t_{1}+\cdots+2r_{\ell-1
}t_{\ell-1}\right)  \left(  \sum_{1\leq i \leq\ell-1} \left(  \left(
2t_{i}\right)  \prod_{j \neq i} r_{j}\right)  \right)  }\\
&  =\dfrac{1}{\left(  2r_{1}t_{1}+\cdots+2r_{\ell-1}t_{\ell-1}\right)  \left(
\dfrac{2t_{1}}{r_{1}}+\cdots+\dfrac{2t_{\ell-1}}{r_{\ell-1}} \right)  }.%
\end{align*} \smallskip

\noindent Consider the following witness candidates for the existentially
quantified variables $r_{1},\ldots,r_{\ell-1}$:
$$r_1=2t_1,\ \ \ r_{i}=\dfrac{1}{2t_{i}}\ \ \ \text{for } 2 \leq i \leq \ell-1.$$

\noindent Obviously, $r_{1},\ldots,r_{\ell-1}>0$. Note
\begin{align*}
C(r)=\  &  \dfrac{1}{\left(  2r_{1}t_{1}+2r_2t_2+\cdots+2r_{\ell-1}t_{\ell-1}\right)  \left(  \dfrac{2t_{1}}{r_{1}}+\dfrac{2t_{2}}{r_{2}}+\cdots+\dfrac{2t_{\ell-1}}{r_{\ell-1}} \right)  }\\
=\  &  \dfrac{1}{\left(2\left(2t_1\right)t_1+ 2\left(  \dfrac
{1}{2t_{2}}\right)  t_{2}+\cdots+2\left(  \dfrac{1}{2t_{\ell-1}}\right)
t_{\ell-1}\right)  \left(  \dfrac{2t_{1}}{2t_{1}}+\dfrac{2t_{2}}{\frac
{1}{2t_{2}}}+\cdots+\dfrac{2t_{\ell-1}}{\frac{1}{2t_{\ell-1}}} \right)  }\\
=\  &  \dfrac{1}{\left( 4t_1^2+1+\cdots+1\right)  \left(1+ 4t_{2}%
^{2}+\cdots+4t_{\ell-1}^{2}\right)  }\\
<\  &  1.
\end{align*}
Hence the candidates are witnesses.
\end{proof}

\noindent Note that the candidates provided above were discovered via guess-and-check in an attempt to manipulate the expression $C(r)$ into a fraction of the form $\dfrac{1}{1+\text{positive terms}}$ where the stated property would be clear. \bigskip

\noindent Finally, we provide the Existence of Radii Lemma, the core argument supporting the All Angles in Quadrant 1 Lemma (Lemma \ref{lem:core}).  One of the biggest obstacles was in determining which coefficients to examine to prove the claim.  You will notice that different coefficients are required based on the number of angles in this quadrant and where they lie.  For example, the coefficients needed to prove the claim when there are two angles $\dfrac{\pi}{3} \leq \theta_1\leq\theta_2<\dfrac{\pi}{2}$ are different from the coefficients needed when $0 < \theta_1<\dfrac{\pi}{3}$ and $\dfrac{\pi}{3} \leq \theta_2<\dfrac{\pi}{2}$.

\begin{lemma}[Existence of Radii]
\label{lem:q1m2} 
$\underset{\ell \geq 2}{\forall}\ \ \underset{0<\theta_{1}\leq
\ldots\leq\theta_{\ell}<\frac{\pi}{2}}{\forall}\ \ \underset{r_{\theta_{1}},\ldots
,r_{\theta_{\ell}}>0}{\exists}\ \ \underset{\deg(g)=s-1}{\underset{g\in\mathbb{R}[x]}{\forall}%
}\ \ \underset{k \in \{s-2,\, 2\ell+s-5,\, 2\ell+s-2\}}{\exists}\ \ c_k<0$

\end{lemma}

\begin{proof} Note
\begin{align*} &\ \underset{\ell \geq 2}{\forall}\ \ \underset{0<\theta_{1}\leq
\ldots\leq\theta_{\ell}<\frac{\pi}{2}}{\forall}\ \ \underset{r_{\theta_{1}},\ldots
,r_{\theta_{\ell}}>0}{\exists}\ \ \underset{\deg(g)=s-1}{\underset{g\in\mathbb{R}[x]}{\forall}%
}\ \ \underset{k \in \{s-2,\, 2\ell+s-5,\, 2\ell+s-2\}}{\exists}\ \ c_k<0 \\
\iff &\ \underset{\ell \geq 2}{\forall}\ \ \underset{1>t_{1}\geq
\ldots\geq t_{\ell}>0}{\forall}\ \ \underset{r_{\theta_{1}},\ldots
,r_{\theta_{\ell}}>0}{\exists}\ \ \underset{\{b_0,\hdots,b_{s-2}\}\in\mathbb{R}}{\forall}\ \ \underset{k \in \{s-2,\, 2\ell+s-5,\, 2\ell+s-2\}}{\exists}\ \ c_k<0
\end{align*}
since the leading coefficient of $g$ is assumed to be 1. \bigskip

\noindent Let $\ell \geq 2$ and $t_{1},\ldots,t_{\ell}\ $be such that $1>t_{1}\geq\cdots\geq
t_{\ell}>0$.  Let $r_{1}>0$. \bigskip

\begin{description}[align=left,labelwidth=1em]
\item[\sf Claim 1:] When $s=2$, $\underset{r_{\ell}^{\left(  1\right)  }>0}{\exists
}\ \ \underset{r_{\ell}\geq r_{\ell}^{\left(  1\right)  }}{\forall}\ \ \underset{b_{0}%
}{\forall}\ \  c_{2\ell+s-5}<0\ \vee\ c_{2\ell+s-2}<0$. 

Note that $s=2$ only when $\ell=2$ and $\dfrac{\pi}{3}\leq \theta_1\leq\theta_2<\dfrac{\pi}{2}$.  Note
\begin{align*}
c_{2\ell+s-5}  &  =a_{0}+a_{1}b_{0}\\
c_{2\ell+s-2}  &  =a_{3}+b_{0}.%
\end{align*}

We have
\[%
\begin{array}
[c]{lclclcl}%
c_{2\ell+s-5}<0 & \iff & b_{0}>-\dfrac{a_{0}}{a_{1}} & \iff & b_{0}>\dfrac
{e_{4}(\alpha_{1},\alpha_{2},\alpha_{3},\alpha_{4})}{e_{3}(\alpha_{1}%
,\alpha_{2},\alpha_{3},\alpha_{4})} & \iff & b_{0}>\dfrac{r_{1}r_{2}}%
{2r_{2}t_{1}+2r_{1}t_{2}}\\
c_{2\ell+s-2}<0 & \iff & b_{0}<-a_{3} & \iff & b_{0}<e_{1}(\alpha_{1}%
,\alpha_{2},\alpha_{3},\alpha_{4}) & \iff & b_{0}<2r_{1}t_{1}+2r_{2}t_{2}.%
\end{array}
\]

Note $\underset{r_{2}>0}{\exists}\ \underset{b_{0}}{\forall}%
\ c_{2\ell+s-5}<0\ \vee\ c_{2\ell+s-2}<0$. Note
\[%
\begin{array}
[c]{crl}
& \underset{b_{0}}{\forall}\ c_{2\ell+s-5}<0\ \vee\ c_{2\ell+s-2}<0 & \\
\Longleftarrow & \dfrac{r_{1}r_{2}}{2r_{2}t_{1}+2r_{1}t_{2}} & <2r_{1}%
t_{1}+2r_{2}t_{2}\\
\iff & \dfrac{r_{1}r_{2}}{(2r_{2}t_{1}+2r_{1}t_{2})(2r_{1}t_{1}+2r_{2}t_{2})}
& <1.
\end{array}
\]
Note
\[
\ \underset{b_{0}}{\forall}\ \lim_{r_{2}\rightarrow\infty} \dfrac{r_{1}r_{2}%
}{(2r_{2}t_{1}+2r_{1}t_{2})(2r_{1}t_{1}+2r_{2}t_{2})} =0
\]
since
\begin{align*}
\deg_{r_{2}}\left(  r_{1}r_{2} \right)   &  =1\\
\deg_{r_{2}}\left(  (2r_{2}t_{1}+2r_{1}t_{2})(2r_{1}t_{1}+2r_{2}t_{2})\right)
&  =2.
\end{align*}

\item Hence $\underset{r_{2}^{\left(  1\right)  }>0}{\exists
}\ \ \underset{r_{2}\geq r_{2}^{\left(  1\right)  }}{\forall}\ \ \underset{b_{0}%
}{\forall}\ \  c_{2\ell+s-5}<0\ \vee\ c_{2\ell+s-2}<0$.  The claim has been proved.

\end{description}

%\begin{example}
\bigskip
\noindent {\sf Example for Claim 1:} {\it 
The following two graphs demonstrate Claim 1 for the degree 4 polynomial
$$f=\left(x^2-2(1)\cos\left(\dfrac{10\pi}{24}\right)x+(1)^2\right)\left(x^2-2r_2\cos\left(\dfrac{11\pi}{24}\right)x+r_2^2\right)$$
where $\dfrac{\pi}{3}<\dfrac{10\pi}{24}<\dfrac{11\pi}{24}<\dfrac{\pi}{2}$.  Note that $r_1=1$ and $g=x+b_0$, as $s=2$.

\begin{center} \includegraphics[width=2in]{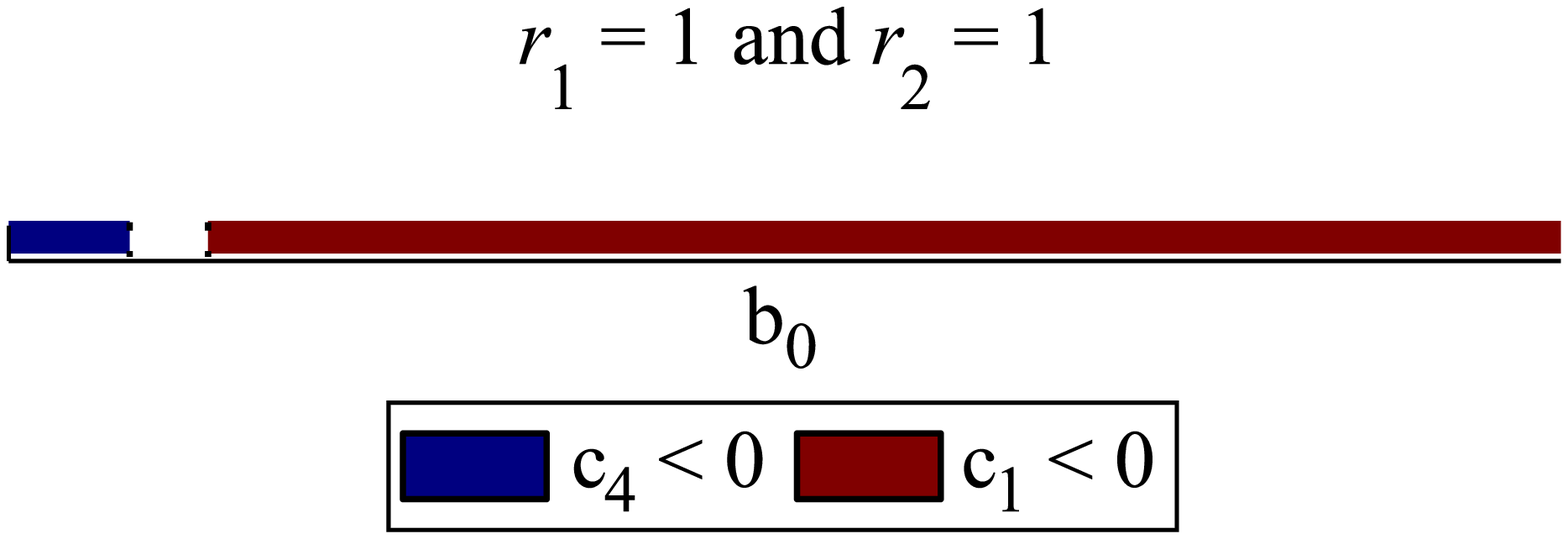} \hspace{1cm} \includegraphics[width=2in]{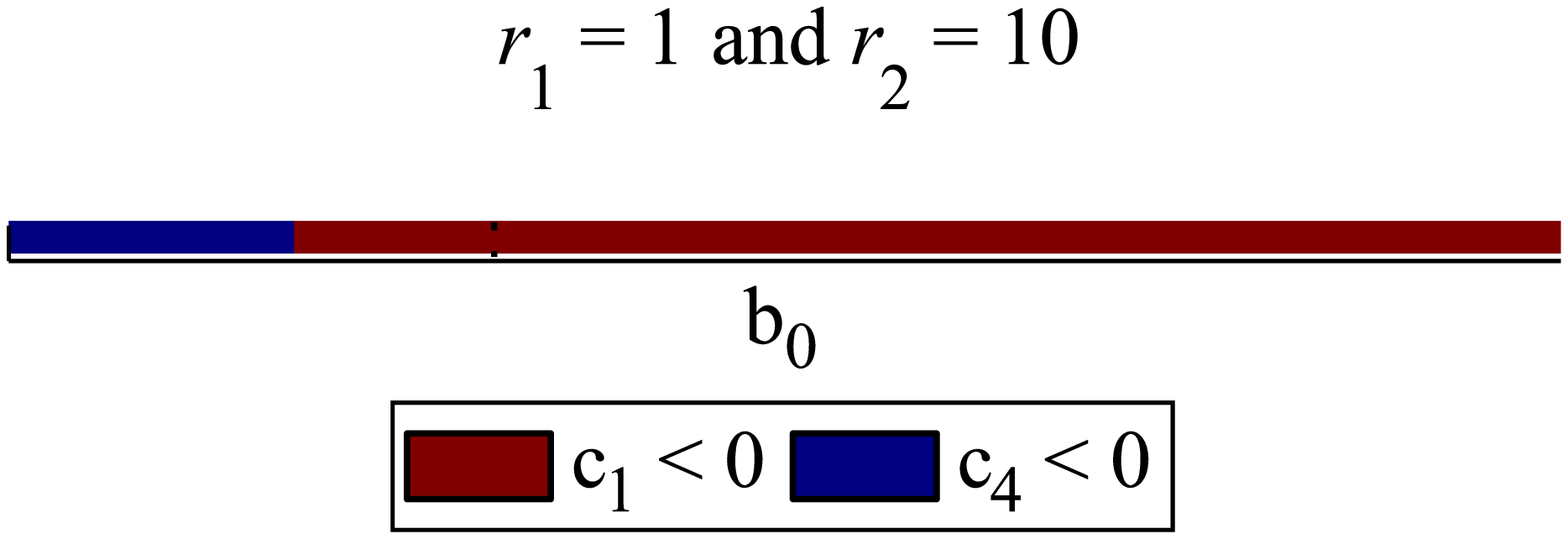} \end{center}
When both radii are equal to 1, there are potential values of $b_0$ that do not produce a negative coefficient.  When $r_2$ is increased to 10, at least one coefficient is negative for any $b_0$.} 
\bigskip

%\end{example}

\noindent For the remainder of the proof, let $s \geq3$.  Let $r_{1},\ldots,r_{\ell-1}>0$ be such that $C(r)<1$ is true, where
$$C(r)=\dfrac{e_{0}\left(  \alpha_{1},\ldots,\alpha_{\ell-1},\alpha
_{\ell+1},\ldots,\alpha_{2\ell-1}\right)  \, e_{2\ell-2}\left( \alpha_{1},\ldots,\alpha_{\ell-1},\alpha
_{\ell+1},\ldots,\alpha_{2\ell-1}\right)
}{e_{1}\left(\alpha_{1},\ldots,\alpha_{\ell-1},\alpha
_{\ell+1},\ldots,\alpha_{2\ell-1}\right)  \, e_{2\ell-3}\left( \alpha_{1},\ldots,\alpha_{\ell-1},\alpha
_{\ell+1},\ldots,\alpha_{2\ell-1}\right)  }.$$
The existence of such radii is justified based on the value of $\ell$.
\begin{enumerate}
\item In the case where $\ell=2$ and $s \geq 3$,
\begin{align*} C(r) & =\dfrac{e_{0}\left(  \alpha_{1},\alpha_{3}\right)  \, e_{2}\left( \alpha_{1},\alpha_{3}\right)
}{e_{1}\left(\alpha_{1},\alpha_{3}\right)  \, e_{1}\left( \alpha_{1},\alpha_{3}\right)  } \\
& =\dfrac{1 \cdot \alpha_1\alpha_3}{(\alpha_1+\alpha_3)(\alpha_1+\alpha_3)} \\
& =\dfrac{r_1^2}{4r_1^2t_1^2}\\
& =\dfrac{1}{4t_1^2} \end{align*}
Since $s \geq 3$, we must have $0<\theta_1<\dfrac{\pi}{3}$ and $t_1>\dfrac{1}{2}$.  Hence, $C(r)<1$.
\item In the case where $\ell \geq 3$, we know such radii exist from Lemma \ref{lem:witness}.
\end{enumerate} \medskip

\noindent Using the coefficient expressions from Lemma \ref{lem:coeffexp}, note
\begin{align*}
c_{s-2}  &  =0\iff b_{s-3}=-\dfrac{a_{0}}{a_{1}}b_{s-2}-\sum_{i=2}^{s-2}%
\dfrac{a_{i}}{a_{1}}b_{(s-2)-i}\\
c_{2\ell+s-5}  &  =0\iff b_{s-3}=-\dfrac{a_{2\ell-3}}{a_{2\ell-2}}%
b_{s-2}-\dfrac{a_{2\ell-1}b_{s-4}+b_{s-5}+a_{2\ell-4}}{a_{2\ell-2}}.
\end{align*} \smallskip

\noindent Let $b_{0},\ldots,b_{s-4}$ be arbitrary but fixed. \bigskip

\noindent Let $L_{k}$ be the line given by $c_{k}=0$ in $\left(b_{s-2},b_{s-3}\right)$ space. Let $\mu_{k}$ be the slope of
$L_{k}$. Then 
\begin{align*}
\mu_{s-2} & =-\dfrac{a_{0}}{a_{1}} \\
\mu_{2\ell+s-5} & =-\dfrac{a_{2\ell-3}}{a_{2\ell-2}}. \end{align*}

\noindent Since $a_{i}=(-1)^{2\ell-i}e_{2\ell-i}\left(  \alpha_{1},\hdots,\alpha
_{2\ell}\right)  $, we have
\begin{align*}
a_{0}  &  =e_{2\ell}\left(  \alpha_{1},\hdots,\alpha_{2\ell}\right) \\
a_{1}  &  =-e_{2\ell-1}\left(  \alpha_{1},\hdots,\alpha_{2\ell}\right) \\
a_{2\ell-3}  &  =-e_{3}\left(  \alpha_{1},\hdots,\alpha_{2\ell}\right) \\
a_{2\ell-2}  &  =e_{2}\left(  \alpha_{1},\hdots,\alpha_{2\ell}\right).
\end{align*}
Then 
\begin{align*}
\mu_{s-2} & =\dfrac{e_{2\ell}\left(  \alpha_{1},\hdots,\alpha_{2\ell
}\right)  }{e_{2\ell-1}\left(  \alpha_{1},\hdots,\alpha_{2\ell}\right)  } \\
\mu_{2\ell+s-5} & =\dfrac{e_{3}\left(  \alpha_{1},\hdots,\alpha_{2\ell}\right)
}{e_{2}\left(  \alpha_{1},\hdots,\alpha_{2\ell}\right)  }. \end{align*}

\begin{description}[align=left,labelwidth=1em]
\item[\sf Claim 2:] $\underset{r_{\ell}^{\left(  2\right)  }>0}{\exists
}\ \underset{r_{\ell}\geq r_{\ell}^{\left(  2\right)  }}{\forall}\ \mu
_{2\ell+s-5}>\mu_{s-2}$. 

Note as $r_{\ell}\rightarrow\infty$%
\begin{align*}
e_{2\ell}\left(  \alpha_{1},\ldots,\alpha_{2\ell}\right)   &  \rightarrow
\alpha_{\ell}\alpha_{2\ell}\ e_{2\ell-2}\left(  \alpha_{1},\ldots,\alpha_{\ell-1},\alpha
_{\ell+1},\ldots,\alpha_{2\ell-1}\right) \\
e_{2\ell-1}\left(  \alpha_{1},\ldots,\alpha_{2\ell}\right)   &  \rightarrow
\alpha_{\ell}\alpha_{2\ell}\ e_{2\ell-3}\left(  \alpha_{1},\ldots,\alpha_{\ell-1},\alpha
_{\ell+1},\ldots,\alpha_{2\ell-1}\right) \\
e_{3}\left(  \alpha_{1},\ldots,\alpha_{2\ell}\right)   &  \rightarrow
\alpha_{\ell}\alpha_{2\ell}\ e_{1}\left(  \alpha_{1},\ldots,\alpha_{\ell-1},\alpha
_{\ell+1},\ldots,\alpha_{2\ell-1}\right) \\
e_{2}\left(  \alpha_{1},\hdots,\alpha_{2\ell}\right)   &  \rightarrow
\alpha_{\ell}\alpha_{2\ell}\ e_{0}\left(  \alpha_{1},\ldots,\alpha_{\ell-1},\alpha
_{\ell+1},\ldots,\alpha_{2\ell-1}\right).
\end{align*}
Then
\begin{align*}
\mu_{s-2}  &  \rightarrow\dfrac{e_{2\ell-2}\left(   \alpha_{1},\ldots,\alpha_{\ell-1},\alpha
_{\ell+1},\ldots,\alpha_{2\ell-1}\right)  }{e_{2\ell
-3}\left(   \alpha_{1},\ldots,\alpha_{\ell-1},\alpha
_{\ell+1},\ldots,\alpha_{2\ell-1}\right)  }\\
\mu_{2\ell+s-5}  &  \rightarrow\dfrac{ e_{1}\left(   \alpha_{1},\ldots,\alpha_{\ell-1},\alpha
_{\ell+1},\ldots,\alpha_{2\ell-1}\right)  }%
{e_{0}\left(   \alpha_{1},\ldots,\alpha_{\ell-1},\alpha
_{\ell+1},\ldots,\alpha_{2\ell-1}\right)  }.%
\end{align*}

Then for sufficiently large $r_{\ell}$,
\[%
\begin{array}
[c]{crl}
& \mu_{2\ell+s-5} & >\mu_{s-2}\\
\iff & \dfrac{ e_{1}\left(   \alpha_{1},\ldots,\alpha_{\ell-1},\alpha
_{\ell+1},\ldots,\alpha_{2\ell-1}\right)  }{e_{0}\left(   \alpha_{1},\ldots,\alpha_{\ell-1},\alpha
_{\ell+1},\ldots,\alpha_{2\ell-1}\right)  } &
>\dfrac{e_{2\ell-2}\left(  \alpha_{1},\ldots,\alpha_{\ell-1},\alpha
_{\ell+1},\ldots,\alpha_{2\ell-1}\right)  }{e_{2\ell-3}\left(  \alpha_{1},\ldots,\alpha_{\ell-1},\alpha
_{\ell+1},\ldots,\alpha_{2\ell-1}\right)  }\\
\iff & 1 & >C(r)
\end{array}
\]
which is true for the $r_{1},\ldots,r_{\ell-1}$ we have previously chosen based on the claim from Lemma \ref{lem:witness}.  Recall that the elementary symmetric polynomials in these roots are always positive by Lemma \ref{lem:epos}.  The claim has been proved.  
\end{description}

\noindent Continuing with the proof of the lemma, let $r_{\ell}^{\left(  2\right)  }$ be such that $\underset{r_{\ell}\geq r_{\ell}^{\left(  2\right)  }}{\forall}\ \mu
_{2\ell+s-5}>\mu_{s-2}$. Such
$r_{\ell}^{\left(  2\right)  }\ $exists due to the previous claim. Let
$r_{\ell}\ $be arbitrary but fixed such that $r_{\ell}\ \geq r_{\ell}^{\left(
2\right)  }$. Then over the space $\left(  b_{s-2},b_{s-3}\right)  $, there exists a unique
intersection point of $L_{s-2}\ $and $L_{2\ell+s-5}$.  Let $\left(
p,q\right)  $ be the intersection point.

%\begin{example}
\bigskip
\noindent {\sf Example of $(p,q)$:} {\it 
In this example, $f=\left(x^2-2(1)\cos\left(\dfrac{7\pi}{24}\right)x+(1)^2\right)\left(x^2-2r_2\cos\left(\dfrac{10\pi}{24}\right)x+r_2^2\right)$, so $s=3$, $g=x^2+b_1x+b_0$, and $(p,q)$ is the intersection point of $L_1$ with $L_2$.
\begin{center} \includegraphics[width=2in]{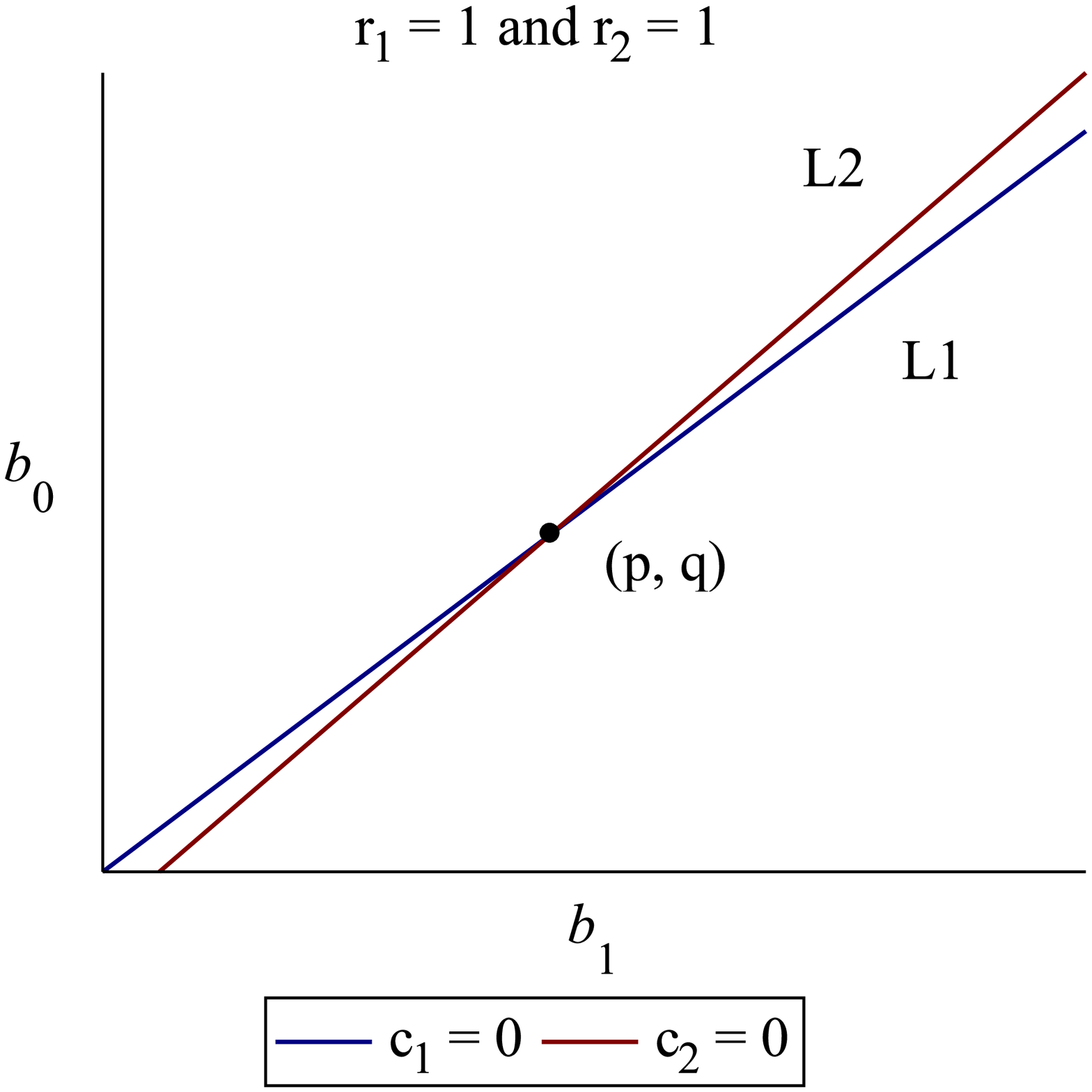} \hspace{1cm} \includegraphics[width=2in]{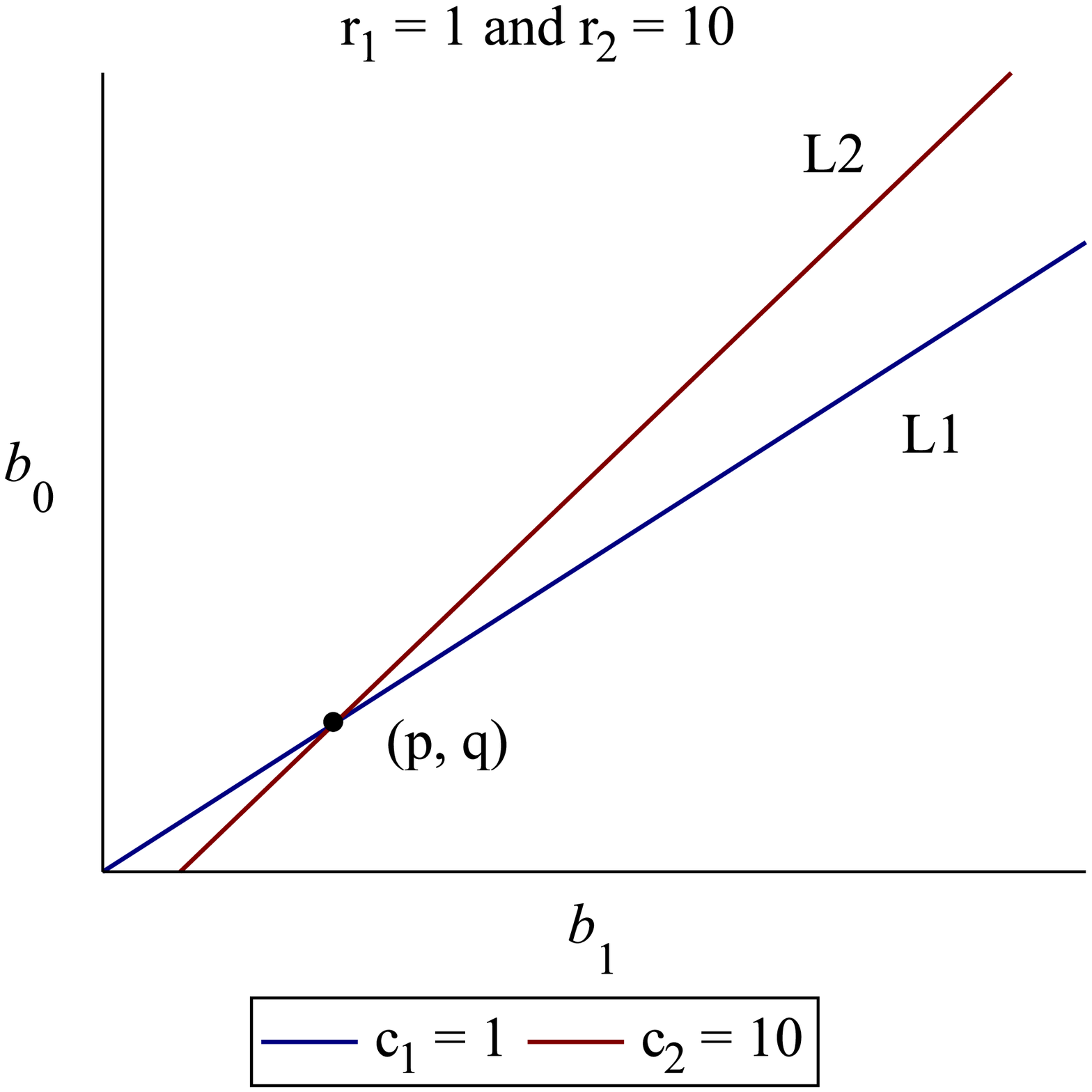} \end{center}
We will continue using this example to illustrate the remainder of the proof.} \bigskip 
%\end{example}

\noindent Continuing with the proof of the lemma, as shown below, we have $p=\dfrac{a_{2\ell-2}d_{0}-a_{1}d_{1}}{a_{0}a_{2\ell-2}%
-a_{1}a_{2\ell-3}}$ and \\ $q=\dfrac{a_{0}d_{1}-a_{2\ell-3}d_{0}}{a_{0}%
a_{2\ell-2}-a_{1}a_{2\ell-3}}$ where $d_{0}=-\displaystyle\sum_{i=2}%
^{s-2}a_{i}b_{(s-2)-i}$ and $d_{1}=-a_{2\ell-1}b_{s-4}-b_{s-5}-a_{2\ell-4}$.  Note%
\[%
\begin{array}
[c]{clc}
& c_{s-2}=0\ \ \ \wedge\ \ \ c_{2\ell+s-5}=0 & \\
\Longleftrightarrow & a_{0}b_{s-2}+a_{1}b_{s-3}+\sum_{i=2}^{s-2}%
a_{i}b_{(s-2)-i}=0 & \wedge\\
& a_{2\ell-4}+a_{2\ell-3}b_{s-2}+a_{2\ell-2}b_{s-3}+a_{2\ell-1}b_{s-4}%
+b_{s-5}=0 & \\
\Longleftrightarrow & a_{0}b_{s-2}+a_{1}b_{s-3}=-\sum_{i=2}^{s-2}%
a_{i}b_{(s-2)-i} & \wedge\\
& a_{2\ell-3}b_{s-2}+a_{2\ell-2}b_{s-3}=-a_{2\ell-1}b_{s-4}-b_{s-5}%
-a_{2\ell-4} & \\
\iff & a_{0}b_{s-2}+a_{1}b_{s-3}=d_{0}\ \ \ \ \text{where }d_{0}=-\sum
_{i=2}^{s-2}a_{i}b_{(s-2)-i} & \wedge\\
& a_{2\ell-3}b_{s-2}+a_{2\ell-2}b_{s-3}=d_{1}\ \ \ \ \text{where }
d_{1}=-a_{2\ell-1}b_{s-4}-b_{s-5}-a_{2\ell-4} & \\
& \  & \\
\Longleftrightarrow & \left[
\begin{array}
[c]{cc}%
a_{0} & a_{1}\\
a_{2\ell-3} & a_{2\ell-2}%
\end{array}
\right]  \left[
\begin{array}
[c]{c}%
b_{s-2}\\
b_{s-3}%
\end{array}
\right]  =\left[
\begin{array}
[c]{c}%
d_{0}\\
d_{1}%
\end{array}
\right]  & \\
& \  & \ \\
\Longleftrightarrow & b_{s-2}=\frac{\left\vert
\begin{array}
[c]{cc}%
d_{0} & a_{1}\\
d_{1} & a_{2\ell-2}%
\end{array}
\right\vert }{\left\vert
\begin{array}
[c]{cc}%
a_{0} & a_{1}\\
a_{2\ell-3} & a_{2\ell-2}%
\end{array}
\right\vert }\ \ \ \wedge\ \ \ b_{s-3}=\frac{\left\vert
\begin{array}
[c]{cc}%
a_{0} & d_{0}\\
a_{2\ell-3} & d_{1}%
\end{array}
\right\vert }{\left\vert
\begin{array}
[c]{cc}%
a_{0} & a_{1}\\
a_{2\ell-3} & a_{2\ell-2}%
\end{array}
\right\vert } & \\
& \  & \ \\
\Longleftrightarrow & b_{s-2}=\dfrac{a_{2\ell-2}d_{0}-a_{1}d_{1}}%
{a_{0}a_{2\ell-2}-a_{1}a_{2\ell-3}}\ \ \ \wedge\ \ \ b_{s-3}=\dfrac{a_{0}%
d_{1}-a_{2\ell-3}d_{0}}{a_{0}a_{2\ell-2}-a_{1}a_{2\ell-3}}. & 
\end{array}
\]

\begin{description}[align=left,labelwidth=1em]
\item[\sf Claim 3:]  $\underset{r_{\ell}\geq r_{\ell}^{\left(  2\right)  }}{\forall
}\ \underset{b_{s-2}>p}{\underset{b_{0},\ldots,b_{s-2}}{\forall}}%
\ c_{s-2}<0\ \vee\ c_{2\ell+s-5}<0$.

Let $r_{\ell}^{\left(  2\right)  }$ be such that $\underset{r_{\ell}\geq
r_{\ell}^{\left(  2\right)  }}{\forall}\ \mu_{2\ell+s-5}>\mu_{s-2}$. In Claim 2, we showed that such $r_{\ell}^{\left(  2\right)  }$ exists.  Let $r_{\ell}\geq r_{\ell}^{\left(  2\right)  }\ $be arbitrary but fixed. \bigskip

Let $b_{0},\ldots,b_{s-4}$ be arbitrary but fixed. We need to show
$\underset{b_{s-2}>p}{\underset{b_{s-2},b_{s-3}}{\forall}}\ c_{s-2}%
<0\ \vee\ c_{2\ell+s-5}<0$.

Over the space $\left(  b_{s-2},b_{s-3}\right)  $ where $b_{s-2}>p$, we have

\begin{enumerate}
\item $c_{2\ell+s-5}=0$ line and $c_{s-2}=0$ line do not intersect

\item $c_{2\ell+s-5}=0$ line is above $c_{s-2}=0$ line.
\end{enumerate}

Let $\left(  b_{s-2},b_{s-3}\right)  $ be an arbitrary but fixed point
such that $b_{s-2}>p$. Then $\left(  b_{s-2},b_{s-3}\right)  \ \,$is above
$L_{s-2}$ or below $L_{2\ell+s-5}$.  Note%
\[%
\begin{array}
[c]{lllll}%
\left(  b_{s-2},b_{s-3}\right)  \ \text{is above }L_{s-2} &
\Longleftrightarrow & b_{s-3}>-\frac{a_{0}}{a_{1}}b_{s-2}+\frac{d_{0}}{a_{1}%
}\ \  & \Longleftrightarrow & c_{s-2}<0\\
\left(  b_{s-2},b_{s-3}\right)  \ \text{is below }L_{2\ell+s-5}\  &
\Longleftrightarrow & b_{s-3}<-\frac{a_{2\ell-3}}{a_{2\ell-2}}b_{s-2}%
+\frac{d_{1}}{a_{2\ell-2}} & \Longleftrightarrow & c_{2\ell+s-5}<0
\end{array}
\]

\noindent since $a_{1}=(-1)^{2\ell-1}e_{2\ell-1}\left(  \alpha_{1},\hdots,\alpha_{2\ell
}\right)  <0$ and $a_{2\ell-2}=(-1)^{2}e_{2}\left(  \alpha_{1},\hdots,\alpha
_{2\ell}\right)  >0$.\bigskip

Thus $c_{s-2}<0$ or $c_{2\ell+s-5}<0$.
 The claim has been proved. 

\end{description}

%\begin{example}
\bigskip 
\noindent {\sf Example of Claim 3:} {\it Again, $f=\left(x^2-2(1)\cos\left(\dfrac{7\pi}{24}\right)x+(1)^2\right)\left(x^2-2r_2\cos\left(\dfrac{10\pi}{24}\right)x+r_2^2\right)$, so $s=3$, \\ $g=x^2+b_1x+b_0$, and $(p,q)$ is the intersection point of $L_1$ with $L_2$. 
\begin{center} \includegraphics[width=2in]{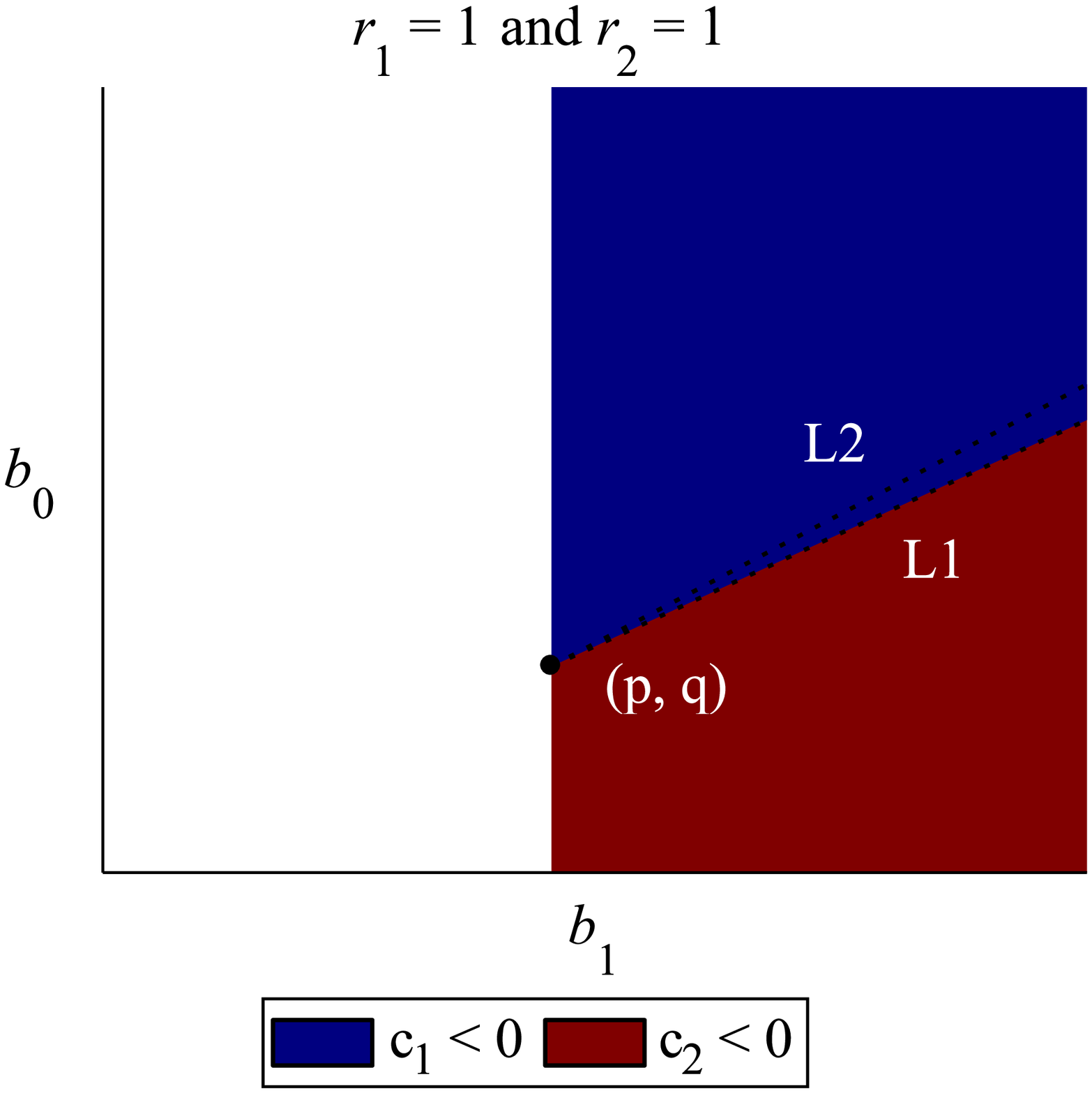} \hspace{1cm} \includegraphics[width=2in]{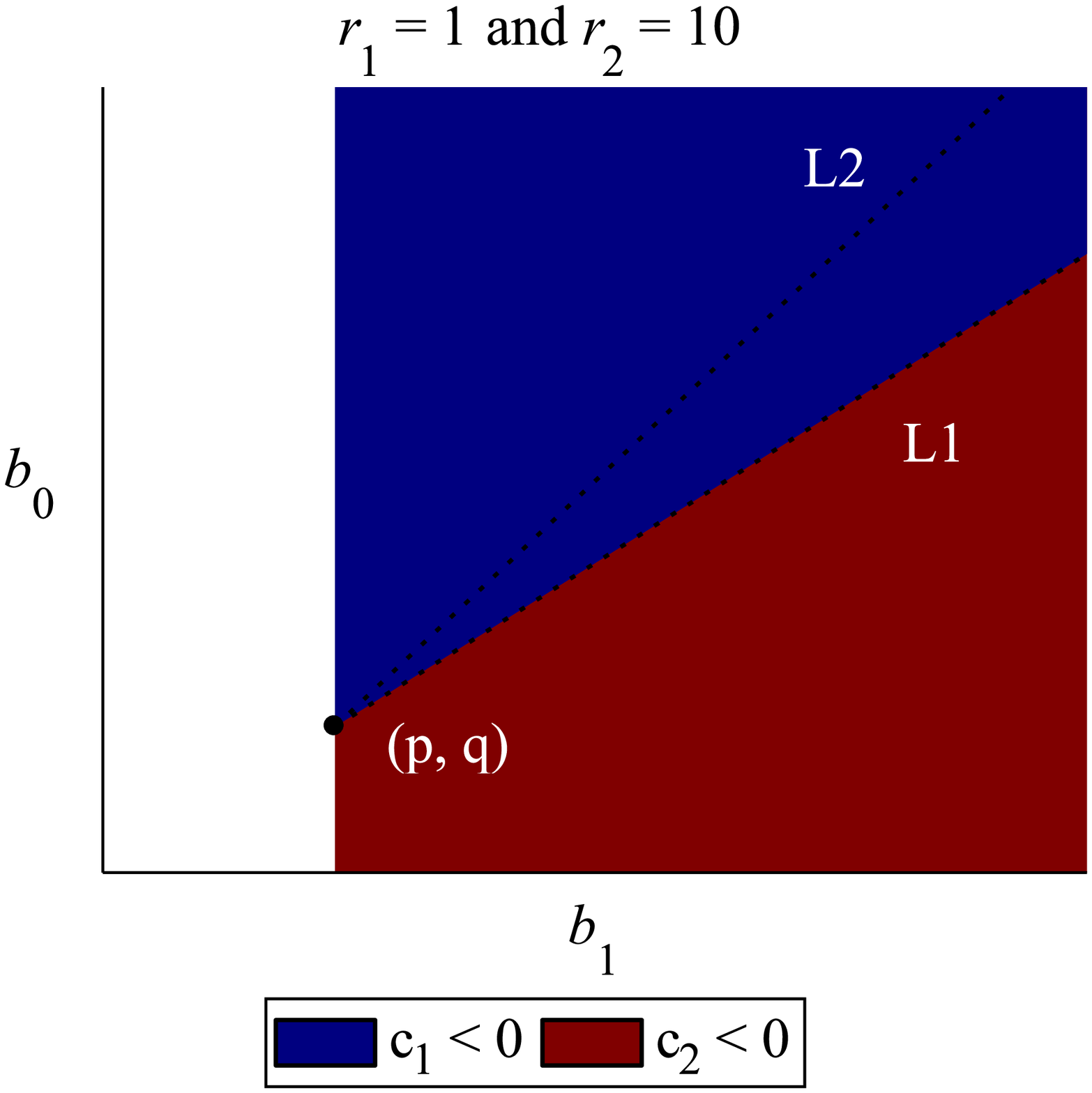} \end{center} 
For any point with $b_{s-2}=b_1>p$, we can see that either $c_1<0$ or $c_2<0$.} \bigskip
%\end{example}

\begin{description}[align=left,labelwidth=1em]
\item[\sf Claim 4:] $\underset{r_{\ell}^{\left(  3\right)  }>0}{\exists
}\ \underset{r_{\ell}\geq r_{\ell}^{\left(  3\right)  }}{\forall
}\ \underset{b_{s-2}\leq p}{\underset{b_{0},\ldots,b_{s-2}}{\forall}%
}\ c_{2\ell+s-2}<0$.

Note%
\[
\underset{r_{\ell}>0}{\forall}\ \underset{b_{s-2}\leq p}{\underset{b_{0}%
,\ldots,b_{s-2}}{\forall}}\ \ \dfrac{a_{2\ell-2}d_{0}-a_{1}d_{1}}{e_{1}\left(
\alpha_{1},\ldots,\alpha_{2\ell}\right)  \left(  a_{0}a_{2\ell-2}%
-a_{1}a_{2\ell-3}\right)  }<1\ \ \Longrightarrow\ \ c_{2+s-2}<0
\]

since%
\[%
\begin{array}
[c]{crl}
& c_{2\ell+s-2} & <0\\
\iff & b_{s-2} & <-a_{2\ell-1}\\
\Longleftarrow & p_{} & <-a_{2\ell-1}\ \ \ \,\text{since }b_{s-2}\leq p\\
\iff & \dfrac{a_{2\ell-2}d_{0}-a_{1}d_{1}}{a_{0} a_{2\ell-2}-a_{1}a_{2\ell-3}}
& <e_{1}\left(  \alpha_{1},\ldots,\alpha_{2\ell}\right) \\
\iff & \dfrac{a_{2\ell-2}d_{0}-a_{1}d_{1}}{e_{1}\left(  \alpha_{1}%
,\ldots,\alpha_{2\ell}\right)  \left(  a_{0}a_{2\ell-2}-a_{1}a_{2\ell
-3}\right)  } & <1.
\end{array}
\]
Once more, $e_{1}\left( \alpha_{1}%
,\ldots,\alpha_{2\ell}\right) >0$ from Lemma \ref{lem:epos}. \bigskip

Let $e_{k}=e_{k}\left(  \alpha_{1},\ldots,\alpha_{2\ell}\right)  $.
Note
\begin{align*}
&  \dfrac{a_{2\ell-2}d_{0}-a_{1}d_{1}}{e_{1}\left(  \alpha_{1},\ldots
,\alpha_{2\ell}\right)  \left(  a_{0}a_{2\ell-2}-a_{1}a_{2\ell-3}\right)  }\\
=\  &  \dfrac{-e_{2}\sum_{i=2}^{s-2}(-1)^{2\ell-i}e_{2\ell-i}b_{(s-2)-i}%
+e_{2\ell-1}\left(  e_{1}b_{s-4}-b_{s-5}-e_{4}\right)  }{e_{1}\left(
e_{2\ell}e_{2}-e_{2\ell-1}e_{3}\right)  }.%
\end{align*}

Note
\[
\ \underset{b_{0},\ldots,b_{s-2}}{\forall}\ \lim_{r_{\ell}\rightarrow\infty
}\dfrac{-e_{2}\sum_{i=2}^{s-2}(-1)^{2\ell-i}e_{2\ell-i}b_{(s-2)-i}+e_{2\ell
-1}\left(  e_{1}b_{s-4}-b_{s-5}-e_{4}\right)  }{e_{1}\left(  e_{2\ell}%
e_{2}-e_{2\ell-1}e_{3}\right)  }=0
\]
since
\begin{align*}
\deg_{r_{\ell}}\left(  -e_{2}\sum_{i=2}^{s-2}(-1)^{2\ell-i}e_{2\ell
-i}b_{(s-2)-i}+e_{2\ell-1}\left(  e_{1}b_{s-4}-b_{s-5}-e_{4}\right)  \right)
&  \leq4\\
\deg_{r_{\ell}}\left(  e_{1}\left(  e_{2\ell}e_{2}-e_{2\ell-1}e_{3}\right)
\right)   &  =5.
\end{align*}
The claim has been proved. 
\end{description}

%\begin{example} 
\bigskip
\noindent {\sf Example of Claim 4:} {\it Once more, $f=\left(x^2-2(1)\cos\left(\dfrac{7\pi}{24}\right)x+(1)^2\right)\left(x^2-2r_2\cos\left(\dfrac{10\pi}{24}\right)x+r_2^2\right)$, so $s=3$, $g=x^2+b_1x+b_0$, and $(p,q)$ is the intersection point of $L_1$ with $L_2$. 
\begin{center} \includegraphics[width=2in]{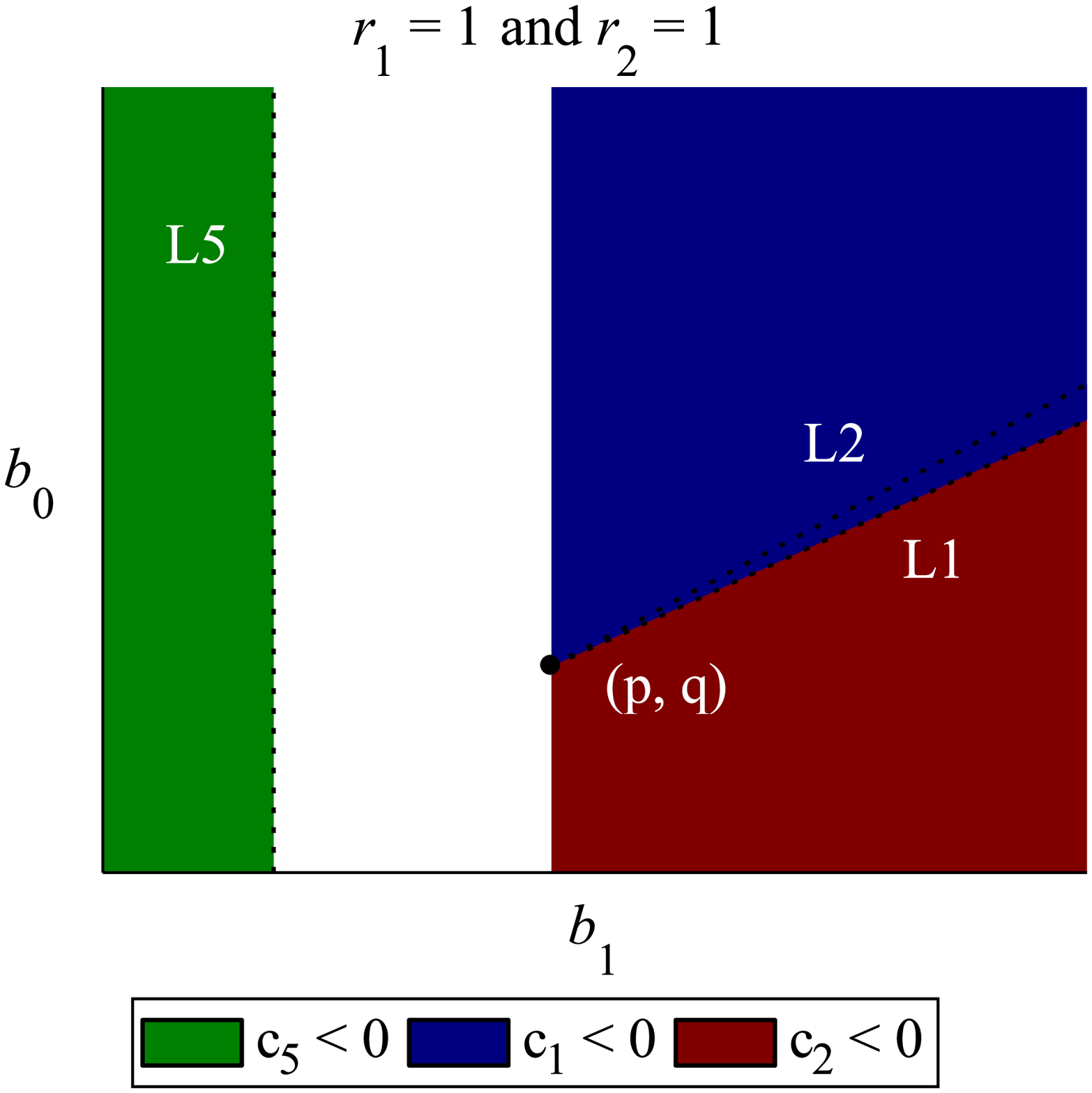} \hspace{1cm} \includegraphics[width=2in]{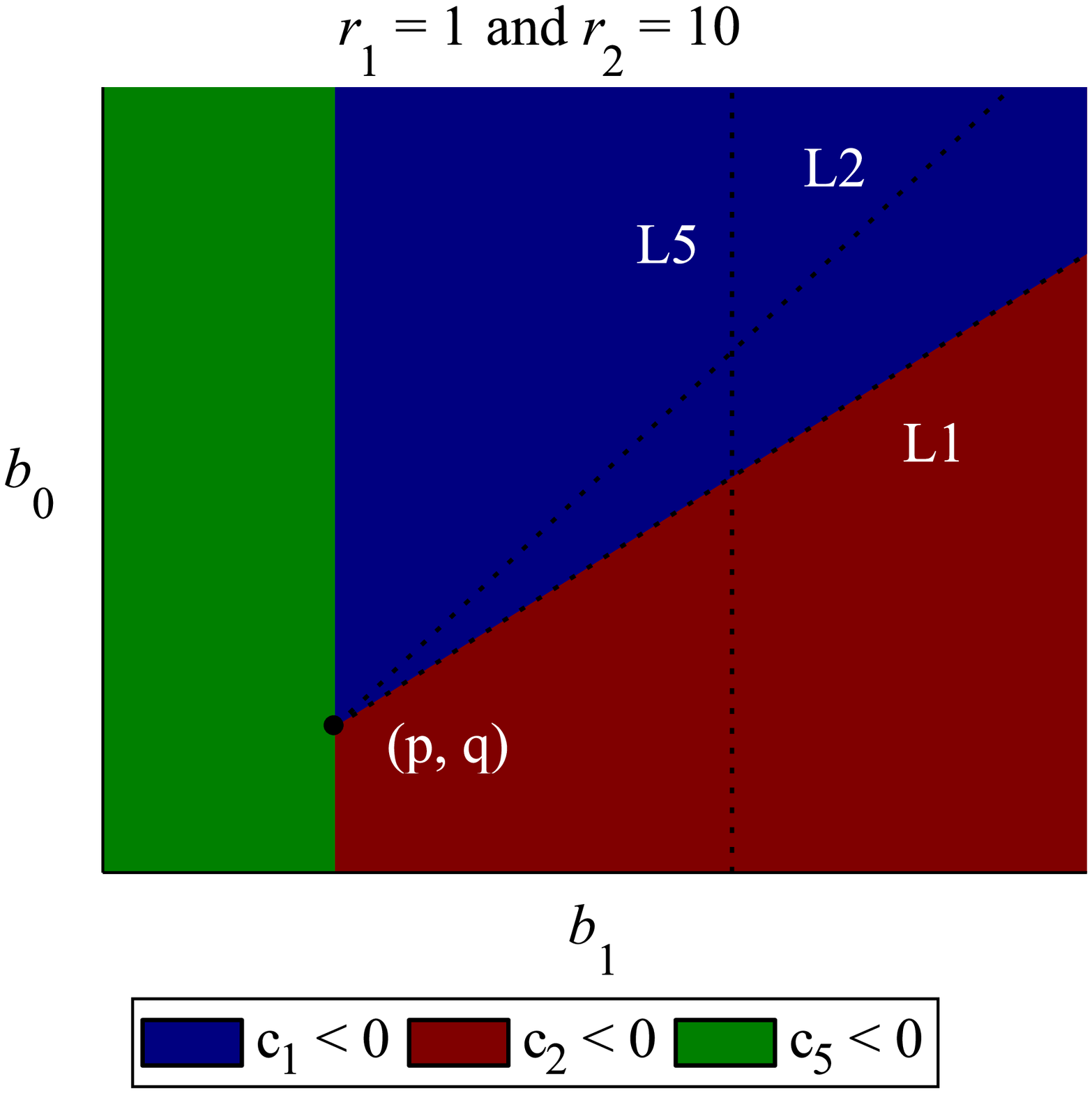} \end{center}
Notice that when $r_2=1$, there is a region of points with $b_{s-2}=b_1 \leq p$ where $c_{2\ell+s-2}=c_5$ is not negative.  However, when $r_2=10$, $c_5<0$ for all points where $b_1 \leq p$.}
\bigskip
% \end{example}

\noindent We continue with the proof of the lemma. \bigskip

\noindent From Claim 1, when $\ell=s=2$, for some $r_{\ell}^{\left(  1\right)  }>0$ we have
$\underset{r_{\ell}>r_{\ell}^{(1)}}{\forall}\ \ \underset{b_{0}}{\forall}\ \ c_{2\ell+s-5}%
<0\ \vee\ c_{2\ell+s-2}<0$.

\noindent From Claim 3, when $s \geq3$, for some $r_{\ell}^{\left(  2\right)  }>0$ we have
\[\underset{r_{\ell}>r_{\ell}^{\left(  2\right)  }}{\forall}\ \ \underset{b_{s-2}%
>p}{\underset{b_{0},\ldots,b_{s-2}}{\forall}}\ \ c_{s-2}<0\ \vee\ c_{2\ell+s-5}<0.\]

\noindent From Claim 4, when $s \geq3$, for some $r_{\ell}^{\left(  3\right)  }>0$ we have
$\underset{r_{\ell}\geq r_{\ell}^{\left(  3\right)  }}{\forall}\ \ \underset{b_{s-2}%
\leq p}{\underset{b_{0},\ldots,b_{s-2}}{\forall}}\ \ c_{2\ell+s-2}<0$.

\noindent Then for $r_{\ell}^{*}=\max\{r_{\ell}^{(1)},r_{\ell}^{(2)},r_{\ell}^{(3)}\}$, we have%

\[
\underset{r_{\ell}\geq r_{\ell}^{*} }{\forall}\ \left[  \left[\underset{b_{s-2}%
>p}{\underset{b_{0},\ldots,b_{s-2}}{\forall}}\ c_{s-2}<0\ \vee\ c_{2\ell+s-5}<0 \right]\ \ \wedge\ \ \left[\underset{b_{s-2}\leq
p}{\underset{b_{0},\ldots,b_{s-2}}{\forall}}\ c_{2\ell+s-5}<0\ \vee\ c_{2\ell
+s-2}<0\right]\right].
\] \smallskip

\noindent Hence
\[
\underset{r_{\ell}>0}{\exists}\ \ \underset{b_{0},\ldots,b_{s-2}}{\forall
}\ \ \underset{k\in\left\{  s-2,\ 2\ell+s-5,\ 2\ell+s-2\right\}  }{\bigvee}%
\ \ c_{k}<0.
\] 
The lemma has finally been proved.
\end{proof}

\subsection{Concerning Angles in Quadrant 2,
$\dfrac{\pi}{2}\leq\phi\leq\pi$}\label{sec:q2}

Finally, we present the proof for the case when some angles fall in quadrant 2, including the axes.  We convert the optimality claim into a statement on the separation of two sets in Lemma \ref{lem:ch}.  Then we prove that for any polynomial for which Curtiss' bound is optimal, there is a radius such that we can multiply by a linear or quadratic factor with angle $\dfrac{\pi}{2}\leq\phi\leq\pi$ and this radius while maintaining the optimality of the bound. Inductively, we can extend this argument to any collection of angles in quadrant 2. 

\begin{center}
\begin{tikzpicture}
[edge from parent fork down]
\node  {Some Angles in Quadrant 2}[sibling distance = 2.5cm, level distance = 1cm]
    child {node {One New Quadrant 2 Factor} 
    child {node {Separation of Sets}}};
\end{tikzpicture}
\end{center}

\noindent The main challenge of this section was in the One New Quadrant 2 Factor Lemma (Lemma \ref{lem:extension}).  To prove a separation between the two necessary sets, we examined the Euclidean distance between them and proved it must be greater than zero, a surprisingly non-trivial task.

\begin{lemma}[Some Angles in Quadrant 2]
\label{lem:reduce}We have%
\[
\underset{\frac{\pi}{2}\leq\phi_{1}\leq\cdots\leq\phi_{k}<\pi}{\forall
}\ \ \underset{0<\theta_{1}\leq\cdots\leq\theta_{\ell}<\frac{\pi}{2}}{\forall
}\ \ \underset{p_{1},\ldots,p_{t}>0}{\exists}\ \ \underset{r_{\phi_{1}}%
,\ldots,r_{\phi_{k}}>0}{\exists}\ \ \underset{r_{\theta_{1}},\ldots
,_{\theta_{\ell}}>0}{\forall}\ \ \operatorname*{opt}\left( f_{\theta,r_{\theta}}
\,f_{\phi,r_{\phi}}\,f_{\pi,p}\right)  =\operatorname*{opt}\left(
f_{\theta,r_{\theta}}\right).
\]

\end{lemma}

\begin{proof}
We will first induct on $k$, the number of quadratic factors with $\dfrac{\pi
}{2} \leq\phi_{i}<\pi$ in $f_{\phi,r_{\phi}}$, to show that
$\operatorname*{opt}\left( f_{\theta,r_{\theta}}\, f_{\phi,r_{\phi}}\right)
=\operatorname*{opt}\left(  f_{\theta,r_{\theta}}\right)  $.

\begin{description}[align=left,labelwidth=1em]
\item[\sf Base Case:] The claim holds for $k=0$. \medskip

It is trivially true.

\item[\sf Hypothesis:] Assume the claim holds for $k$ quadratic factors with $\dfrac{\pi
}{2} \leq\phi_{i}<\pi$.
\[
\underset{r_{\phi_{1}}, \ldots, r_{\phi_{k}}>0}{\exists}\ \ \underset{r
_{\theta_{1}},\ldots,r_{\theta_{\ell}}>0}{\forall}\ \ \operatorname*{opt}%
\left( f_{\theta,r_{\theta}}\, f_{\phi,r_{\phi}}\right)  =\operatorname*{opt}%
\left(  f_{\theta,r_{\theta}}\right)
\]

\item[\sf Induction Step:] Consider $k+1$ quadratic factors with $\dfrac{\pi}{2}
\leq\phi_{i}<\pi$. \medskip

Assume $r_{\phi_{1}},\ldots,r_{\phi_{k}}>0$ are
such that the induction hypothesis holds. By the One More Quadrant 2 Factor Lemma (Lemma \ref{lem:extension}),
\begin{align*}
\underset{r_{\phi_{k+1}}>0}{\exists}\ \ \underset{r_{\theta_{1}},\ldots
,_{\theta_{\ell}}>0}{\forall}   &  \operatorname*{opt}\left(
f_{\theta,r_{\theta}}\, f_{\phi,r_{\phi}}\, \left(  x^{2}-2r_{\phi_{k+1}}\cos\phi_{k+1}\,x+r_{\phi_{k+1}}^{2}\right)\right) \\
&  =\operatorname*{opt}\left( f_{\theta,r_{\theta}}\, f_{\phi,r_{\phi}}\right)+\operatorname*{opt}\left(  x^{2}-2r_{\phi_{k+1}}\cos\phi_{k+1}%
\,x+r_{\phi_{k+1}}^{2}\right).  
\end{align*}

By the Curtiss Bound, (Theorem \ref{thm:bound}), $\operatorname*{opt}\left(  x^{2}-2r_{\phi_{k+1}%
}\cos\phi_{k+1}\,x+r_{\phi_{k+1}}^{2}\right)  =0$, so we have
\begin{align*}
\operatorname*{opt}\left(
f_{\theta,r_{\theta}}\, f_{\phi,r_{\phi}}\, \left(  x^{2}-2r_{\phi_{k+1}}\cos\phi_{k+1}\,x+r_{\phi_{k+1}}^{2}\right)\right)   &  =\operatorname*{opt}\left(  f_{\theta,r_{\theta}}\, f_{\phi,r_{\phi}}\right) \\
&  =\operatorname*{opt}\left(  f_{\theta,r_{\theta}}\right).
\end{align*}
Hence, we have $\underset{r_{\phi_{1}}, \ldots, r_{\phi_{k}}>0}{\exists
}\ \ \underset{r _{\theta_{1}},\ldots,r_{\theta_{\ell}}>0}{\forall
}\ \ \operatorname*{opt}\left( f_{\theta,r_{\theta}}\, f_{\phi,r_{\phi}}\right)  =\operatorname*{opt}\left(  f_{\theta,r_{\theta}}\right)  $.
\end{description} \bigskip

Now we will induct on $t$, the number of linear factors in $f_{\pi,p}$, to
show that
\[
\underset{p_{1},\ldots,p_{t}>0}{\exists}\ \ \underset{r_{\phi_{1}}, \ldots,
r_{\phi_{k}}>0}{\exists}\ \ \underset{r _{\theta_{1}},\ldots,r_{\theta_{\ell}%
}>0}{\forall}\ \ \operatorname*{opt}\left(  f_{\theta,r_{\theta}}
\,f_{\phi,r_{\phi}}\,f_{\pi,p}\right)  =\operatorname*{opt}\left(  f_{\theta
,r_{\theta}}\right).
\]

\begin{description}[align=left,labelwidth=1em]
\item[\sf Base Case:] The claim holds for $t=0$. \medskip

Trivially true.

\item[\sf Hypothesis:] Assume the claim holds for $t$ linear factors.
\[
\underset{p_{1},\ldots,p_{t}>0}{\exists}\ \ \underset{r_{\phi_{1}}, \ldots,
r_{\phi_{k}}>0}{\exists}\ \ \underset{r _{\theta_{1}},\ldots,r_{\theta_{\ell}%
}>0}{\forall}\ \ \operatorname*{opt}\left(  f_{\theta,r_{\theta}}
\,f_{\phi,r_{\phi}}\,f_{\pi,p}\right)  =\operatorname*{opt}\left(  f_{\theta
,r_{\theta}}\right)
\]

\item[\sf Induction Step:] Consider $t+1$ linear factors. \medskip

Assume
$p_{1},\ldots,p_{t},r_{\phi_{1}}, \ldots, r_{\phi_{k}}>0$ are such that the
induction hypothesis holds. By the One More Quadrant 2 Factor Lemma (Lemma \ref{lem:extension}),
\[
\underset{p_{t+1}>0}{\exists}\ \ \underset{r _{\theta_{1}},\ldots
,r_{\theta_{\ell}}>0}{\forall}\ \ \operatorname*{opt}\left(  f_{\theta,r_{\theta}}
\,f_{\phi,r_{\phi}}\,f_{\pi,p}\, \left(
x+p_{t+1}\right)\right)
=\operatorname*{opt}\left(
f_{\theta,r_{\theta}}
\,f_{\phi,r_{\phi}}\,f_{\pi,p}\right)+\operatorname*{opt}\left(  x+p_{t+1}\right).  
\]
By the Curtiss Bound (Theorem \ref{thm:bound}), $\operatorname*{opt}\left(  x+p_{t+1}\right)
=0$, so we have
\begin{align*}
\operatorname*{opt}\left(  f_{\theta,r_{\theta}}
\,f_{\phi,r_{\phi}}\,f_{\pi,p}\, \left(
x+p_{t+1}\right)\right)   &  =\operatorname*{opt}\left(
f_{\theta,r_{\theta}}
\,f_{\phi,r_{\phi}}\,f_{\pi,p}\right) \\
&  =\operatorname*{opt}\left(  f_{\theta,r_{\theta}}\right).
\end{align*}
Hence, we have $\underset{p_{1},\ldots,p_{t}>0}{\exists}\ \ \underset{r_{\phi
_{1}}, \ldots, r_{\phi_{k}}>0}{\exists}\ \ \underset{r _{\theta_{1}}%
,\ldots,r_{\theta_{\ell}}>0}{\forall}\ \ \operatorname*{opt}\left(  f_{\theta,r_{\theta}}
\,f_{\phi,r_{\phi}}\,f_{\pi,p}\right)  =\operatorname*{opt}%
\left(  f_{\theta,r_{\theta}}\right)  $.
\end{description}
\end{proof}

\noindent The above proof relies on the One More Quadrant 2 Factor Lemma (Lemma \ref{lem:extension}).  To prove this lemma, we will build additional notation and convert the optimality claim to a claim about separation of sets.  
Recall that%
\[
A_{s}=\left[
\begin{array}
[c]{cccccc}%
a_{0} & \cdots & \cdots & a_{n} &  & \\
& \ddots &  &  & \ddots & \\
&  & a_{0} & \cdots & \cdots & a_{n}%
\end{array}
\right]  \in\mathbb{R}^{\left(  s+1\right)  \times(s+n+1)}%
\]
and, from the Coefficients Lemma (Lemma \ref{lem:bA}), $\operatorname*{coeffs}\left(  gf\right)  =bA_{s}$.  We partition $A_{s}$ into two sub-matrices as $A_{s}=\left[
L_{s}|R_{s}\right]  $ where%
\[
L_{s}=\left[
\begin{array}
[c]{ccc}%
a_{0} & \cdots & a_{n-1}\\
& \ddots & \vdots\\
&  & a_{0}\\
&  & \\
&  & \\
&  &
\end{array}
\right]  \in\mathbb{R}^{\left(  s+1\right)  \times n}\ \ \ \text{and }%
R_{s}=\left[
\begin{array}
[c]{cccccc}%
a_{n} &  &  &  &  & \\
\vdots & \ddots &  &  &  & \\
\vdots &  & \ddots &  &  & \\
a_{0} &  &  & \ddots &  & \\
& \ddots &  &  & \ddots & \\
&  & a_{0} & \cdots & \cdots & a_{n}%
\end{array}
\right]  \in\mathbb{R}^{\left(  s+1\right)  \times\left(  s+1\right)  }.%
\]
Let%
\begin{align*}
c  &  =bR_{s}\ \ \in\mathbb{R}^{1\times\left(  s+1\right)  }\\
T_{s}  &  =R_{s}^{-1}L_{s}\ \ \in\mathbb{R}^{\left(  s+1\right)  \times n}.%
\end{align*}

\begin{lemma}[Separation of Sets]
\label{lem:ch}We have%
\[
\underset{g\neq0,\ \deg\left(  g\right)  \leq s}{\exists}%
\ \operatorname*{coeffs}\left(  gf\right)  \geq0\ \ \ \ \Longleftrightarrow
\ \ \ \ \operatorname*{ConvexHull}\left(  T_{s}\right)  \cap\mathbb{R}_{\geq
0}^{n}\neq\emptyset
\]
where $T_{s}$ is viewed as a set of row vectors.
\end{lemma}

\begin{proof}
Note that
\begin{align*}
\operatorname*{coeffs}\left(  gf\right) & =bA_{s}\ \ \ \ \ \text{(from Lemma \ref{lem:bA})}  \\
&  =b\left(  R_{s}R_{s}^{-1}\right)  A_{s}\\
&  =\left(  bR_{s}\right)  \left(  R_{s}^{-1}A_{s}\right) \\
&  =\left(  bR_{s}\right)  \left(  R_{s}^{-1}\left[  L_{s}|R_{s}\right]
\right) \\
&  =\left(  bR_{s}\right)  \left[  R_{s}^{-1}L_{s}|R_{s}^{-1}R_{s}\right] \\
&  =\left(  bR_{s}\right)  \left[  R_{s}^{-1}L_{s}|I\right] \\
&  =c\left[  T_{s}|I\right]  \ \ \ \ \ \ \ \ \text{where\ }c=bR_{s}%
\ \ \,\text{and\ }T_{s}=R_{s}^{-1}L_{s}.%
\end{align*}

\noindent Thus we have
\begin{align*}
&  \ \ \underset{g\neq0,\ \deg\left(  g\right)  \leq s}{\exists}%
\ \operatorname*{coeffs}\left(  gf\right)  \geq0\\
\Longleftrightarrow &  \ \ \underset{c\neq0}{\exists}\ c\left[  T_{s}%
|I\right]  \geq0\\
\Longleftrightarrow &  \ \ \underset{c\neq0}{\exists}\ cT_{s}\geq
0\ \ \text{and }c\geq0\\
\Longleftrightarrow &  \ \ \underset{c\geq0,\ c\neq0}{\exists}\ cT_{s}\geq0\\
\Longleftrightarrow &  \ \ \underset{c\geq0,\ c_{0}+\cdots+c_{s}=1}{\exists
}\ cT_{s}\geq0\\
\Longleftrightarrow &  \ \ \operatorname*{ConvexHull}\left(  T_{s}\right)
\cap\mathbb{R}_{\geq0}^{n}\neq\emptyset.
\end{align*}
\end{proof}

\begin{notation}
Note

\begin{itemize}
\item $h_{\phi,r}=x+r$ when $\phi=\pi$

\item $h_{\phi,r}=x^{2}-2r\cos\phi\,x+r^{2}$ when $\dfrac{\pi}{2} \leq \phi <\pi$

\item $P=\left\{  f\in\mathbb{R}[x]:f(x)>0\ \ \text{for}\ \ x\geq0\right\}  $
\end{itemize}
\end{notation}

\begin{lemma}[One New Quadrant 2 Factor]
\label{lem:extension} $\underset{f\in P,\ \deg(f)\geq1}{\forall
}\ \ \underset{\frac{\pi}{2}\leq\phi \leq \pi}{\forall}$
$\ \ \underset{r>0}{\exists}\ \ \operatorname*{opt}\left( f\,  h_{\phi,r}\right)
=\operatorname*{opt}\left(
f\right)+\operatorname*{opt}\left(  h_{\phi,r}\right)  $
\end{lemma}

\begin{proof}
Let $f\in P$ and $h=f\, h_{\phi,r}$. We need to show $\underset{r>0}{\exists
}\ \operatorname*{opt}(h)=\operatorname*{opt}(f)$ since $\operatorname*{opt}%
(h_{\phi,r})=0$. We will divide the proof into several claims.

\begin{description}[align=left,labelwidth=1em]
\item[\sf Claim 1:] $\underset{r>0}{\exists}\ \operatorname*{opt}%
(h)=\operatorname*{opt}(f)$ $\ \ \ \ \Longleftarrow
\ \ \ \ \underset{r>0}{\exists}\ \operatorname*{CH}(T_{h,0}^{\ast}%
,\ldots,T_{h,s-1}^{\ast})\cap\mathbb{R}_{\geq0}^{n}=\emptyset$

where

\begin{itemize}
\item $n=\deg\left(  f\right)  $

\item $s=\operatorname*{opt}(f)$

\item $d=\deg(h)$

\item The $T_{h,i}$ are the rows of $T_{s-1}$ for $h$.

\item $T_{h,i}^{\ast}$ is obtained from $T_{h,i}$ by deleting the first $d$ elements.
\end{itemize}

Note%
\[%
\begin{array}
[c]{cl}
& \ \operatorname*{opt}(h)=\operatorname*{opt}(f)\\
\iff & \ \operatorname*{opt}(h)\geq s\ \ \ \ \ \text{since}%
\ \ \operatorname*{opt}(h)\leq\operatorname*{opt}(f)+\operatorname*{opt}%
(h_{\phi,r})=\operatorname*{opt}(f)=s\\
\iff & \ \lnot\underset{g\neq0,\ \deg(g)<s}{\exists}\operatorname*{coeffs}%
(g\, h)\geq0\\
\iff & \ \lnot\left(  \operatorname*{CH}(T_{h,0},\ldots,T_{h,s-1}%
)\cap\mathbb{R}_{\geq0}^{n+d}\neq\emptyset\right)\ \ \ \ \text{from Lemma } \ref{lem:ch}  \\
\iff & \ \operatorname*{CH}(T_{h,0},\ldots,T_{h,s-1})\cap\mathbb{R}_{\geq
0}^{n+d}=\emptyset\\
\Longleftarrow & \ \operatorname*{CH}(T_{h,0}^{\ast},\ldots,T_{h,s-1}^{\ast
})\cap\mathbb{R}_{\geq0}^{n}=\emptyset.
\end{array}
\] \medskip

\item[\sf Claim 2:] $\underset{r>0}{\exists}\ \operatorname*{opt}%
(h)=\operatorname*{opt}(f)$ $\ \ \ \ \Longleftarrow
\ \ \ \ \underset{r>0}{\exists}\varepsilon_{h}\left(  r\right)  >0$

where $\varepsilon_{h}\left(  r\right)  $ stands for the minimum Euclidean
distance between$\ \operatorname*{CH}(T_{h,0}^{\ast},\ldots,T_{h,s-1}^{\ast
})\ $and $\mathbb{R}_{\geq0}^{n}$,$\ $that is,%
\[
\varepsilon_{h}(r)\ :=\ \min_{\substack{x\in\operatorname*{CH}(T_{h,0}^{\ast
},\ldots,T_{h,s-1}^{\ast})\\y\in\mathbb{R}_{\geq0}^{n}}}\Vert x-y\Vert\ .
\]

Immediate from the above claim. \medskip

\item[\sf Claim 3:] $\varepsilon_{h}\left(  r\right)  $ is continuous at $r=0$. \medskip

Note
\begin{align*}
\varepsilon_{h}(r)  &  =\ \min_{\substack{c\in\mathbb{R}_{\geq0}^{s}%
\\c_{0}+\cdots+c_{s-1}=1\\y\in\mathbb{R}_{\geq0}^{n}}}\ \left\Vert \sum
_{i=0}^{s-1}c_{i}T_{h,i}^{\ast}\ -\ y\right\Vert \\
&  =\ \min_{\substack{c\in\mathbb{R}_{\geq0}^{s}\\c_{0}+\cdots+c_{s-1}%
=1\\y\in\mathbb{R}_{\geq0}^{n}}}\ \left\Vert \left[  \sum_{i=0}^{s-1}%
c_{i}T_{h,i,1}^{\ast}\ -\ y_{1}\ \ ,\ldots,\ \ \sum_{i=0}^{s-1}c_{i}%
T_{h,i,n}^{\ast}\ -\ y_{n}\right]  \right\Vert \\
&  =\ \underset{z_{0}+\cdots+z_{s-1}=1}{\underset{z\in\mathbb{R}_{\geq0}%
^{s+n}}{\min}}\ \left\Vert \left[  \sum_{i=0}^{s-1}z_{i}T_{h,i,1}^{\ast
}\ -\ z_{s}\ \ ,\ldots,\ \ \sum_{i=0}^{s-1}z_{i}T_{h,i,n}^{\ast}%
\ -\ z_{s+n-1}\right]  \right\Vert \ \ \ \ \text{where }z=(c,y)\\
&  =\ \underset{z\in C}{\min}\ p\left(  z,r\right)
\end{align*}
where $p(z,r)$ is Euclidean distance and
\begin{align*}
C  &  =\left\{  z\in\mathbb{R}_{\geq0}^{s+n}:z_{1}+\cdots+z_{s}=1\right\}  .
\end{align*}

Note, since $p(z,r)$ is the distance between two closed sets, the
minimum distance is realized by a point in each set. Hence,
\[
\varepsilon_{h}(r)=\underset{z\in C}{\min}\ p\left(  z,r\right)
=\underset{z\in C}{\inf}\ p\left(  z,r\right)  .
\]

Note the following:

\begin{enumerate}
\item By Section 3.1.5 of \cite{B04}, $p(z,r)$ is a convex function since
$p(z,r)$ is a norm.

\item By Section 3.2.5 of \cite{B04}, $\varepsilon_{h}(r) =\underset{z\in
C}{\inf}\ p\left(  z,r\right)  $ is a convex function, since $C$ is convex and
$p(z,r)$ is bounded below.

\item By Corollary 3.5.3 in \cite{NP06}, since $\varepsilon_{h}(r)$ is a
convex function defined on a convex set $\mathbb{R}$, $\varepsilon_{h}(r)$ is
continuous on the relative interior of $\mathbb{R}$, $\operatorname*{ri}%
\left(  \mathbb{R}\right)  =\mathbb{R}$.

\item Hence, $\varepsilon_{h}(r)$ is continuous at $r=0$.
\end{enumerate} \medskip

\item[\sf Claim 4:] $\varepsilon_{h}\left(  0\right)  >0$. \medskip

The claim follows immediately from the following  subclaims.
Let $r=0$. Then $h=x^{d} \cdot f$.  
\begin{description}[align=left]
\item[\sf Subclaim 1]: $T_{h,i,j}^{\ast}=T_{f,i,j}$. Note that these are the entries
of $T_{h,s-1}^{*}$ and $T_{f,s-1}$. We will prove the claim using the fact
that $T_{s-1}=R_{s-1}^{-1}L_{s-1}$ for any $f$. \bigskip

Note $R_{f,s-1}=R_{h,s-1}$. This is clear from the definition of
$R_{s-1}$. Hence $R_{f,s-1}^{-1}=R_{h,s-1}^{-1}$. \bigskip

Note that $L_{f,i,j}=L_{h,i,j+d}$. This is clear from the definition of
$L_{s-1}$, since $L_{h,s-1}$ is composed of $L_{f,s-1}$ with $d$ columns of
zeroes added on the left. \bigskip

Note that for $i=0,\hdots ,s-1$ and $j=0,\hdots ,n-1$,
\begin{align*}
T_{h,i,j}^{*}  &  =T_{h,i,j+d}\\
&  =\sum_{k=0}^{s-1} R_{h,i,k}^{-1}L_{h,k,j+d}\\
&  =\sum_{k=0}^{s-1} R_{f,i,k}^{-1}L_{f,k,j}\\
&  =T_{f,i,j}.%
\end{align*}
Hence $T_{h,i,j}^{\ast}=T_{f,i,j}$. \bigskip

\item[\sf Subclaim 2:] $\varepsilon_{h}(0)=\varepsilon_{f}$ where $\varepsilon_{f}$ stands for the minimum Euclidean distance
between$\ \operatorname*{CH}(T_{f,0},\ldots,T_{f,s-1})\ $and $\mathbb{R}%
_{\geq0}^{n}$,$\ $that is,%
\[
\varepsilon_{f} :=\ \min_{\substack{x\in\operatorname*{CH}(T_{f,0}%
,\ldots,T_{f,s-1})\\y\in\mathbb{R}_{\geq0}^{n}}}\Vert x-y\Vert.
\]

To see this, note
\begin{align*}
\varepsilon_{h}(0)  &  =\min_{\substack{c\in\mathbb{R}_{\geq0}^{s}%
\\c_{0}+\cdots+c_{s-1}=1\\y\in\mathbb{R}_{\geq0}^{n}}}\ \left\Vert \sum
_{i=0}^{s-1}c_{i}T_{h,i}^{\ast}\ -\ y\right\Vert \\
&  =\min_{\substack{c\in\mathbb{R}_{\geq0}^{s}\\c_{0}+\cdots+c_{s-1}%
=1\\y\in\mathbb{R}_{\geq0}^{n}}}\ \left\Vert \sum_{i=0}^{s-1}c_{i}%
T_{f,i}\ -\ y\right\Vert \\
&  =\min_{\substack{x\in\operatorname*{CH}(T_{f,0},\ldots,T_{f,s-1}%
)\\y\in\mathbb{R}_{\geq0}^{n}}}\Vert x-y\Vert\\
&  =\varepsilon_{f}.%
\end{align*}

\item[\sf Subclaim 3:] $\varepsilon_{f}>0$. \medskip

\noindent Note since $\operatorname*{opt}\left(  f\right)  =s$, we have
$\operatorname*{CH}(T_{f,0},\ldots,T_{f,s-1})\cap\mathbb{R}_{\geq0}%
^{n}=\emptyset$. Thus $\varepsilon_{f}>0$. 
\end{description} \medskip
\end{description} \bigskip

\noindent From the above four claims, we immediately have%
\[
\underset{r>0}{\exists}\ \operatorname*{opt}(h)=\operatorname*{opt}(f).
\]
Thus we have proved the lemma.
\end{proof}

%\[%
%\underset{\tau_{\ell+1}>0}{\exists}\operatorname*{opt}\left(  f_{\left(
%\theta_{1},\ldots,\theta_{\ell}\right)  ,\left(  \tau_{1},\ldots,\tau_{\ell
%}\right)  }h_{\theta_{\ell+1},\tau_{\ell+1}}\right)  =\operatorname*{opt}%
%\left(  f_{\left(  \theta_{1},\ldots,\theta_{\ell}\right)  ,\left(  \tau
%_{1},\ldots,\tau_{\ell}\right)  }\right)  +\operatorname*{opt}\left(
%h_{\theta_{\ell+1},\tau_{\ell+1}}\right)
%\]

\begin{example} Consider the following function, for which $b(f)=4$: 
$$f=\prod_{i=1}^3\left(x^2-2(10^{i-1})\cos(\theta_i)\, x+(10^{i-1})^2\right)
\;\;\text{where}\;\; \theta=\left\{\dfrac{7\pi}{24},\dfrac{10\pi}{24},\dfrac{11\pi}{24}\right\}.$$

\noindent When $h_{\phi,r}=x^2-2(10^0)\cos\left(\dfrac{14\pi}{24}\right)x+(10^0)^2$, we have $\operatorname*{opt}(f\,  h_{\phi,r})=3$.  
\\

\noindent When $h_{\phi,r}=x^2-2(10^3)\cos\left(\dfrac{14\pi}{24}\right)x+(10^3)^2$, we have $\operatorname*{opt}(f\,  h_{\phi,r})=b(f)=4$.\end{example}

\begin{example} Consider the following function, for which $b(f)=4$: 
$$f=\prod_{i=1}^4\left(x^2-2(10^{i-1})\cos(\theta_i)\, x+(10^{i-1})^2\right)
\;\;\text{where}\;\; \theta=\left\{\dfrac{7\pi}{24},\dfrac{10\pi}{24},\dfrac{11\pi}{24},\dfrac{14\pi}{24}\right\}.$$

\noindent  When $h_{\phi,r}=x+10^0$, we have $\operatorname*{opt}(f\,  h_{\phi,r})=3$.
\\
  
\noindent  When $h_{\phi,r}=x+10^4$, we have $\operatorname*{opt}(f\,  h_{\phi,r})=4$.
\end{example}

\medskip\noindent{\bf Acknowledgements}. Hoon Hong's work was 
supported by National Science Foundation (CCF: 2212461 and 1813340). 

\bibliographystyle{abbrv}
\bibliography{paper}

\end{document}